\documentclass[onecolumn]{IEEEtran}
\usepackage{amsmath,amsfonts,amssymb,euscript, graphicx, 
epsfig,
enumerate,float,afterpage, subfigure, ifthen}%

\usepackage{url}
\usepackage{hyperref}
\usepackage{algpseudocode}

\newtheorem{thm}{Theorem}
\newtheorem{cor}{Corollary}
\newtheorem{lem}{Lemma}

\newcommand{\expect}[1]{\mathbb{E}\left[#1\right]}

\newcommand{\norm}[1]{||{#1}||}

\newcommand{\script}[1]{{{\cal{#1} }}}

\begin{document}

\title{Adaptive Optimization for Stochastic Renewal Systems}

\author{Michael J. Neely \\ University of Southern California\\ \url{https://viterbi-web.usc.edu/~mjneely/}
}


\maketitle


\begin{abstract} 
This paper considers online optimization for a system that performs a sequence of back-to-back tasks. Each task can be processed in one of 
multiple processing modes that affect the duration of the task, the reward earned, and an additional vector of penalties (such as energy or cost). Let $A[k]$ be a random matrix of parameters that specifies the duration, reward, and penalty vector under each processing option for task $k$. The goal is to observe $A[k]$ at the start of each new task $k$ and then choose a processing mode for the task so that, over time, time average reward is maximized subject to time average penalty constraints. This is a \emph{renewal optimization problem} and is challenging because the probability distribution for the $A[k]$ sequence is unknown.  Prior work shows that any algorithm that comes within $\epsilon$ of optimality must have
$\Omega(1/\epsilon^2)$ convergence time.  The only known algorithm that can meet this bound operates without 
time average penalty constraints and uses a diminishing stepsize that cannot adapt when probabilities change. This paper develops a new algorithm that is adaptive and comes within $O(\epsilon)$ of optimality for any interval of $\Theta(1/\epsilon^2)$ tasks over which  probabilities are held fixed, regardless of probabilities before the start of the interval. 
\end{abstract} 

\section{Introduction}

This paper considers online optimization for a system that performs a sequence of tasks (Fig. \ref{fig:tasks}). Each new task starts when the previous one ends.  At the start of each new task $k \in \{1, 2, 3, ...\}$, a matrix $A[k]$ of parameters about the task is revealed. Initially, we assume the matrices $\{A[k]\}_{k=1}^{\infty}$ are independent and identically distributed (i.i.d.) over tasks (this is eventually relaxed to assume the i.i.d. property holds only over a finite block of $m$ consecutive tasks).  
 The controller observes $A[k]$ and then makes a decision about how the task should be processed. The decision, together with  $A[k]$, determines the duration of the task, the reward earned by processing that task, and a vector of additional penalties. Specifically, define 
\begin{align*}
T[k] &= \mbox{duration of task $k$} \\
R[k] &= \mbox{reward of task $k$} \\
Y[k] &=(Y_1[k], \ldots, Y_n[k]) \\
&=\mbox{penalty vector for task $k$}
\end{align*}
where $n$ is a fixed positive integer.   The duration of each task is assumed to be lower bounded by some value $t_{min}>0$.

Each matrix $A[k]$ is assumed to have  finite size $M[k]\times (n+2)$, where $M[k]$ is the number of rows and can change over different tasks $k$. 
The value $M[k]$ is a positive integer that is determined by 
the number of task processing options for task $k$. Each row $r\in \{1, \ldots,M[k]\}$ of matrix $A[k]$ is a row vector of size $n+2$ of the form: 
\begin{equation} \label{eq:row-form}
[T_r[k], R_r[k], Y_{r,1}[k], ..., Y_{r,n}[k]]
\end{equation} 
which represents the duration, reward, and penalties for task $k$ if processing option $r$ is chosen.  The goal is to make a sequence of decisions  that maximizes  time average reward per unit time subject to time average constraints on the penalties. This problem is called a \emph{renewal optimization problem} \cite{renewal-opt-tac}\cite{sno-text}. The problem is challenging because the probability distribution associated with the random matrices $A[k]$ is unknown.  Without a-priori probability information, it is shown in \cite{neely-renewal-jmlr} that any online algorithm that provides an $\epsilon$-approximate solution over tasks $\{1, \ldots, m\}$ must have $m\geq \Omega(1/\epsilon^2)$. An algorithm is developed in \cite{neely-renewal-jmlr} that achieves this optimal asymptotic convergence time for the special case when there are no penalties $Y[k]$.  The algorithm in \cite{neely-renewal-jmlr} uses a Robbins-Monro iteration with a 
vanishing stepsize $\eta[k]=\Theta(1/k)$. The algorithm  assumes $\{A[k]\}$ is i.i.d. forever, starting with task $1$, and cannot adapt if probabilities change at 
some unknown time in the timeline. 
In contrast, the current paper develops a novel algorithm that is \emph{adaptive}.  While our algorithm does not use Robbins-Monro iterations, it can be viewed as having a \emph{constant} stepsize that leads to optimal convergence guarantees over any block of  $\Theta(1/\epsilon^2)$ tasks on which the system probabilities are fixed.  The algorithm also allows for general penalty processes $Y[k]$.

Specifically, imagine an extended situation where independent $\{A[k]\}$ matrices are generated according to 
one distribution for tasks $\{1, \ldots, m_1\}$, another distribution for tasks $\{m_1+1, \ldots, m_2\}$, and another distribution for tasks $\{m_2+1, \ldots, m_3\}$. If the points of change $m_1, m_2, m_3$ were known (and if there were no $Y[k]$ penalties), the Robbins-Monro algorithm of \cite{neely-renewal-jmlr} could be used and its stepsize could be  reset at the times $m_1, m_2, m_3$ to achieve desirable performance over each block.  However, if the times $m_1, m_2, m_3$ are unknown and the stepsize is either never reset, or is reset at incorrect times, then performance over the blocks $\{m_1+1, \ldots, m_2\}$ and $\{m_2+1, \ldots, m_3\}$ can be far from optimal.  The algorithm of the current paper can run over the infinite time horizon, yet 
provides analytical guarantees over \emph{any finite block of consecutive tasks} for which $\{A[k]\}$ is i.i.d. with some (unknown) distribution, regardless of system history before the start of the block. 
Thus, analytical properties over each one of the separate blocks $\{1, \ldots, m_1\}$, $\{m_1+1, \ldots, m_2\}$, and $\{m_2+1, \ldots, m_3\}$ are obtained even when  $m_1, m_2, m_3$ are unknown to the algorithm. 

\begin{figure}[t]
   \centering
   \includegraphics[height=1.3in]{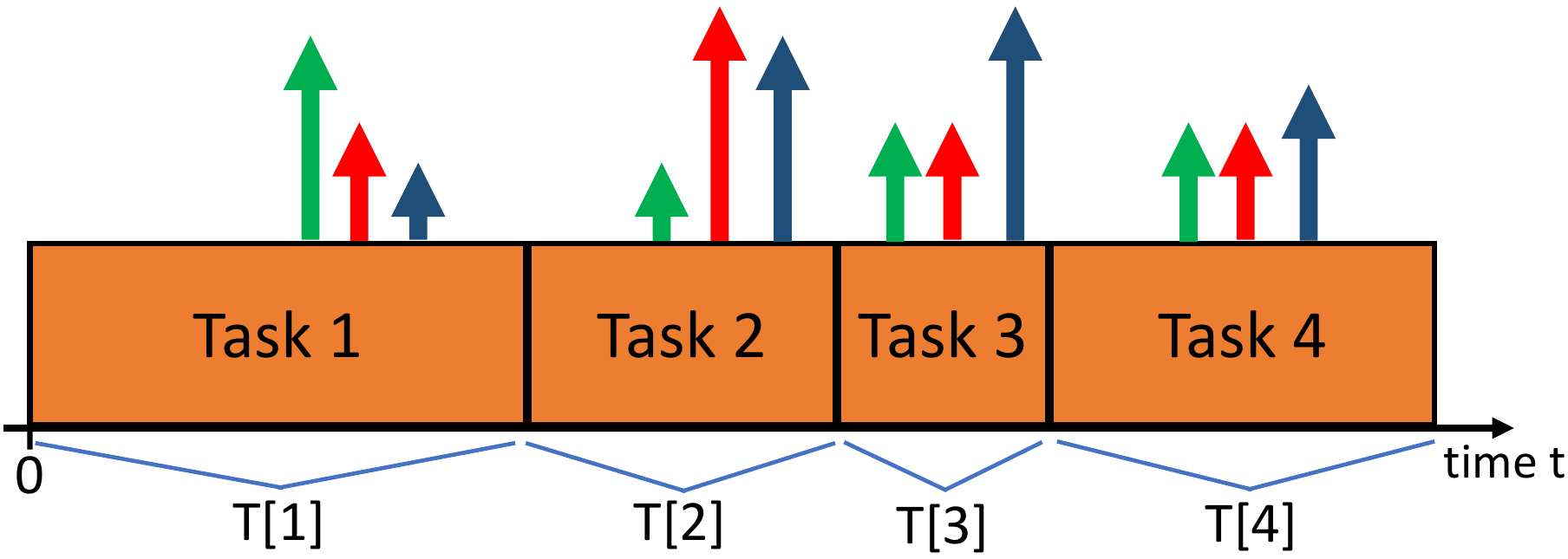} 
   \caption{Four sequential tasks in the timeline. Vertical arrows for each task $k$ represent values for reward $R[k]$ and penalty $Y[k]$. In this example, green is reward (profit), red is energy, blue is quality. The duration of task $k$ and the height of its arrows depend on choices made at the start of task $k$.}
   \label{fig:tasks}
\end{figure}

\subsection{Example application} 

This renewal optimization problem has numerous applications, including 
video processing, image classification, transportation scheduling, and wireless multiple access. For example, consider a device that performs back-to-back image classification tasks with the goal of maximizing time average profit subject to a time average power constraint of $p_{av}$ and an average per-task quality constraint of  $q_{av}$. 
For each task, the device chooses between one of three classification algorithms, each having a different duration of time and yielding certain profit, energy, and quality characteristics. Let $C[k]$ be a $3 \times 4$ 
matrix of parameters for task $k$, such as
\begin{equation} \label{eq:example-matrix}
C[k]=\begin{array}{ccccc} 
\mbox{duration} & \mbox{profit} & \mbox{energy} & \mbox{quality} &  \\
5.1 & 3.6 & 0.3 & 2.0 & \mbox{alg 1}\\
10.2 & 2.8 & 0.7 & 3.0 & \mbox{alg 2}\\
2.7 & 3.0 & 0.2 & 1.0 & \mbox{alg 3}
\end{array}
\end{equation} 
Choosing a classification algorithm for task $k$ reduces to choosing a row of $C[k]$. If the controller choses row 1 then 
$$ T[k]=5.1, R[k]=3.6, Energy[k]=0.3, Quality[k]=2.0$$
The average quality constraint has units of quality/task, while the average power constraint has units of energy/time. To consistently enforce both constraints, we can define penalties $Y_1[k]$ and $Y_2[k]$ for each task $k$ by 
\begin{align}
Y_1[k] &= q_{av} - Quality[k] \label{eq:Y1}\\
Y_2[k] &= Energy[k] - p_{av}T[k] \label{eq:Y2}
\end{align}
Then 
\begin{align*}
&\left(\limsup_{m\rightarrow\infty}\frac{1}{m}\sum_{k=1}^mY_1[k] \leq 0\right) \implies \liminf_{m\rightarrow\infty}\frac{1}{m}\sum_{k=1}^m Quality[k] \geq q_{av} \\
&\left(\limsup_{m\rightarrow\infty}\frac{1}{m}\sum_{k=1}^m Y_2[k]\leq 0 \right) \implies \limsup_{m\rightarrow\infty}\frac{\sum_{k=1}^mEnergy[k]}{\sum_{k=1}^mT[k]} \leq p_{av} 
\end{align*}
where we recall that $T[k]\geq t_{min}$ for all $k$ so there are no divide-by-zero issues.  

We have expressed this example using a matrix $C[k]$ with columns that represent duration, profit, energy, and quality. This can be equivalently represented by a matrix $A[k]$ where the last two columns of $C[k]$ are transformed to the corresponding $Y_1[k]$ and $Y_2[k]$ values, so that for row $r$ of matrix $A[k]$ we have 
\begin{align*}
&Y_{r,1}[k]=q_{av} - Quality_r[k]\\
&Y_{r,2}[k]=Energy_r[k]-p_{av}T_r[k]
\end{align*}
where $Quality_r[k]$ and $T_r[k]$ are taken from row $r$ of the matrix $C[k]$. 

Under a given policy and for each positive integer $m$, define $\overline{T}[m]$ as the empirical average duration per task, averaged over the first $m$ tasks: 
$$\overline{T}[m] = \frac{1}{m}\sum_{k=1}^m T[k]$$
Define $\overline{R}[m]$ and 
$\overline{Y}[m]=(\overline{Y}_1[m], \overline{Y}_2[m])$ similarly.  The time average profit over the first $m$ tasks is
$$ \frac{R[1]+R[2]+\ldots+R[m]}{T[1]+T[2]+\ldots+T[m]} = \frac{\overline{R}[m]}{\overline{T}[m]}$$
Suppose we can make decisions to ensure $\overline{R}[m], \overline{T}[m], \overline{Y}_i[m]$ converge to some constants $\overline{R}, \overline{T}, \overline{Y}_i$ with probability 1 as $m\rightarrow\infty$. The problem of maximizing time average profit subject to the desired constraints can be informally described as: 
\begin{align}
\mbox{Maximize:} \quad & \overline{R}/\overline{T} \label{eq:ex1}\\
\mbox{Subject to:} \quad &\overline{Y}_i\leq 0 \quad \forall i \in \{1,2\}\label{eq:ex2}\\
& (T[k], R[k], Y[k]) \in Row(A[k]) \quad \forall k \in \{1, 2, 3, \ldots\} \label{eq:ex3}
\end{align}
where $Row(A[k])$ denotes the set of rows of $A[k]$. This description illustrates our goals, but is informal because 
because it implicitly requires the sample path limits to exist with probability 1.   A more precise optimization is posed in the next subsection and a closely related deterministic problem is in Section \ref{section:deterministic}.

For wireless multiple access applications, the variable durations of time 
relate to transmission times to the uplink, which depend on code selection and power allocation, while penalties relate to transmission reliability and power expenditure. For ridesharing applications, 
the variable task lengths represent transportation times and the reward is the profit earned by the driver.

\subsection{Convergence time} 

It is assumed that $(T[k], R[k], Y[k])$ is a bounded random vector with a well defined expectation (see boundedness assumptions in 
Section \ref{bounded}). 
It is convenient to work with expectations rather than sample paths. This is similar to the treatment in \cite{renewal-opt-tac}\cite{neely-renewal-jmlr}. For a general scenario with $n$ penalties $Y[k]=(Y_1[k], \ldots, Y_n[k])$, consider the problem 
\begin{align}
\mbox{Maximize:} \quad &\limsup_{m\rightarrow\infty} \frac{\expect{\overline{R}[m]}}{\expect{\overline{T}[m]}} \label{eq:p1}\\
\mbox{Subject to:} \quad &\limsup_{m\rightarrow\infty} \expect{\overline{Y}_i[m]}\leq 0 \quad \forall i \in \{1,\ldots, n\}\label{eq:p2}\\
& (T[k], R[k], Y[k]) \in Row(A[k]) \quad \forall k \in \{1, 2, 3, \ldots\} \label{eq:p3}
\end{align}
The problem is assumed to be feasible, 
meaning that it is possible to satisfy the constraints \eqref{eq:p2}-\eqref{eq:p3}. Let $\theta^*$ denote the optimal objective in \eqref{eq:p1}. Fix $\epsilon>0$. A decision policy is said to be an $\epsilon$-approximation with \emph{convergence time $d$} if 
\begin{align}
&\frac{\expect{\overline{R}[m]}}{\expect{\overline{T}[m]}} \leq \theta^* + \epsilon \quad \forall m \geq d \label{eq:converge1}\\
&\expect{\overline{Y}_i[m]} \leq \epsilon \quad \forall i \in \{1, \ldots, n\} \quad \forall m \geq d \label{eq:converge2}
\end{align}
A decision policy is said to be an $O(\epsilon)$-approximation (with convergence time $d$) if all appearances of $\epsilon$ in the above definition are  replaced by some constant multiple of $\epsilon$.   Given any $\epsilon>0$, the algorithm of this paper produces an $O(\epsilon)$-approximation with convergence time $1/\epsilon^2$ over any interval of $1/\epsilon^2$ tasks.  When operated over an infinite horizon, we show that sample path time averages are similar to the time average expectations. 

\subsection{Prior work} 

The fractional structure of the objective \eqref{eq:p1} is qualitatively similar to a \emph{linear fractional program} \cite{boyd-convex}\cite{fox-linear-fractional-mdp}. 
A nonlinear change of variables in \cite{boyd-convex} shows how to convert a linear fractional program into a convex program. A different nonlinear change of variables is used in \cite{fox-linear-fractional-mdp}. The work \cite{fox-linear-fractional-mdp} uses the method for offline design of an optimal solution to a Markov decision problem. There,  the linear fractional structure relates to minimizing a time average per unit time, similar to our fractional objective \eqref{eq:p1}.  However, the offline computational 
methods of \cite{boyd-convex}\cite{fox-linear-fractional-mdp} cannot be directly used for our online problem. That is 
because time averages are not preserved under nonlinear transformations.  Related work in \cite{asynchronous-markov}\cite{neely-fractional-markov-allerton2011} treats offline and online control for opportunistic Markov decision problems where states include random perturbations similar to the $A[k]$ parameters of the current paper. 
Data  center applications of renewal optimization are in \cite{xiaohan-datacenter}. Applications to general asynchronous systems are 
in \cite{xiaohan-asynch-renewal}.

The renewal optimization problem \eqref{eq:ex1}-\eqref{eq:ex3} is first posed in \cite{renewal-opt-tac} (see also Chapter 7 of \cite{sno-text}).  The solution in \cite{renewal-opt-tac} 
constructs \emph{virtual queues} for each time average inequality constraint \eqref{eq:ex2} and 
makes a decision for each task $k$ to minimize a \emph{drift-plus-penalty ratio}: 
\begin{equation} \label{eq:DPP-ratio}
 \frac{\expect{\tilde{\Delta}[k] -vR[k]}}{\expect{T[k]}} 
 \end{equation} 
where $\Delta[k]$ is the change in a Lyapunov function on the virtual queues; $\tilde{\Delta}[k]$ is a 
simplified version of $\Delta[k]$ that neglects second order terms; and $v>0$ is a parameter that affects accuracy. An exact minimization of the ratio of expectations in \eqref{eq:DPP-ratio} cannot be done unless the probability distribution for the $A[k]$ states is known.  A method for approximating the minimization of \eqref{eq:DPP-ratio} is given in \cite{renewal-opt-tac} based on sampling the $A[k]$ values over a window of previous tasks, although only a partial convergence analysis is given there.   This prior work is based on the Lyapunov drift and max-weight scheduling methods developed  by Tassiulas and Ephremides  for fixed timeslot queueing systems \cite{tass-server-allocation}\cite{tass-radio-nets}.

A different approach in \cite{neely-renewal-jmlr}  uses a \emph{Robbins-Monro iteration} for a special case problem that seeks only to maximize time average reward  $\overline{R}/\overline{T}$ (with no penalties $Y[k]$).    The policy of \cite{neely-renewal-jmlr} chooses $(T[k], R[k]) \in Row(A[k])$ for each task $k$ to minimize $R[k] - \theta[k-1] T[k]$, where $\theta[k-1]$ is an estimate of $\theta^*$ that is updated at the completion of each task according to the Robbins-Monro iteration
$$ \theta[k] = \theta[k-1] + \eta[k](R[k] - \theta[k-1]T[k])$$
where $\eta[k]$ is a stepsize. See \cite{robbins-monro} for the original Robbins-Monro algorithm and \cite{borkar-book}\cite{stochastic-approx-book}\cite{kushner-stochastic-approx}\cite{proximal-robbins-monro}\cite{SGD-robust}\cite{robbins-monro-binary} for extensions in other contexts.
This  approach is desirable because it does not require sampling from a window of past $A[k]$ values. Further, under a particular vanishing stepsize rule, 
the optimality gap of the algorithm is shown to decrease like $O(1/\sqrt{k})$, which is also shown to be asymptotically 
optimal \cite{neely-renewal-jmlr}.   However, it is unclear how to extend the Robbins-Monro technique to handle 
time average penalties $Y[k]$ as considered in the current work. Further, while the vanishing stepsize method in \cite{neely-renewal-jmlr} enables fast convergence,  it makes increasing investments in the probability model and cannot adapt if the system probabilities change. For example, if  $A[k]$ is sampled from a single probability distribution for the first $10^4$ tasks, the $\theta[k]$ value quickly converges to a value near the optimal $\theta^*$. Now suppose that, starting with task $10^4+1$, nature switches the probability distribution (without informing the algorithm of this change). The vanishing stepsize makes it difficult for $\theta[k]$ to change to what it needs for the new distribution.  It may take an additional $10^4$ tasks before $\theta[k]$ starts to approach the new $\theta^*$ value needed for efficient decisions on the new distribution.  The analysis in \cite{neely-renewal-jmlr} shows a fixed stepsize rule is better for adaptation but has a slower convergence time of  $O(1/\epsilon^3)$.

Fixed stepsizes are known to enable adaptation in other contexts. For online convex optimization, 
Zinkevich shows in 
\cite{online-convex} that a fixed stepsize enables regret to be within $O(\epsilon)$ of optimality (as compared to the best fixed decision in hindsight) \emph{over any sequence of $\Theta(1/\epsilon^2)$ steps}.    For adaptive estimation, a recent work \cite{robbins-monro-fixed-stepsize-arxiv} 
considers the problem of removing bias from Markov-based samples. The work \cite{robbins-monro-fixed-stepsize-arxiv}  develops an adaptive 
Robbins-Monro technique that averages between two fixed stepsizes.  Adaptive algorithms are also of recent interest in convex bandit problems, see \cite{Luo2022AdaptiveBC}\cite{haipeng-bandit1}.

\subsection{Our contributions} 

We develop a new algorithm for renewal optimization that, unlike \cite{renewal-opt-tac}, does not require probability information or sampling from the past. The new algorithm has explicit convergence guarantees and 
meets the optimal asymptotic convergence time bound of \cite{neely-renewal-jmlr}. Unlike the Robbins-Monro algorithm of \cite{neely-renewal-jmlr}, our new algorithm allows for general time average penalty constraints. Furthermore, our algorithm is \emph{adaptive} and achieves performance within $O(\epsilon)$ of optimality over any sequence of $\Theta(1/\epsilon^2)$ tasks. This fast adaptation is enabled by using  a new hierarchical decision structure for each task $k$: At the start of task $k$, a max-weight rule is used to choose $(T[k], R[k], Y[k]) \in Row(A[k])$; At the end of task $k$, an \emph{auxiliary variable} is updated to guide the system towards maximized time average reward. Care is taken to ensure the auxiliary variable varies slowly enough (at the timescale of the desired adaptation) so that maximizing the desired fractional objective can be accurately approximated by maximizing a nonfractional objective.

\subsection{Notation} 

This paper often refers to a  
collection of consecutive equations by the first and last equation number separated by a hyphen.  For example, 
``optimization problem \eqref{eq:p1}-\eqref{eq:p3}'' refers to the equations \eqref{eq:p1}, \eqref{eq:p2}, \eqref{eq:p3}.  For vectors $x, y \in \mathbb{R}^n$ we use inner product $x^{\top}y=\sum_{i=1}^n x_iy_i$ and Euclidean norm $\norm{x} = \sqrt{\sum_{i=1}^nx_i^2}$.

\section{Preliminaries}

\subsection{Boundedness assumptions} \label{bounded}

Assume there are nonnegative constants $t_{min}, t_{max}, r_{max}, c, y_{i,min}, y_{i,max}$ (with $t_{min}>0$)  
such that for all $k \in \{1, 2, 3, \ldots\}$, all  $A[k]$, and all possible choices of $(T[k], R[k], Y[k]) \in Row(A[k])$, the following boundedness assumptions surely hold: 
\begin{align}
&t_{min}\leq T[k]\leq t_{max} \label{eq:bound1} \\
&0\leq R[k] \leq r_{max} \label{eq:bound2}\\
&\norm{Y[k]}\leq c \label{eq:bound3}\\
&-y_{i,min}\leq Y_i[k]\leq y_{i,max} \quad \forall i \in \{1, \ldots, n\} \label{eq:bound4} 
\end{align}
where $\norm{Y[k]}$ is the Euclidean norm of $Y[k]=(Y_1[k], \ldots, Y_n[k])$.  
 
 Constraint \eqref{eq:bound2} assumes all rewards $R[k]$ are nonnegative. This is without loss of generality: If the system can have  negative rewards in some bounded interval  $-r_{min}\leq R[k]\leq r_{max}$, where $r_{min}$ is some given positive constant, we define a new \emph{nonnegative} reward 
 $$G[k] = R[k] + (r_{min}/t_{min})T[k]$$  
 The objective of maximizing $\overline{G}/\overline{T}$ is the same as the objective of maximizing $\overline{R}/\overline{T}$. The new reward  $G[k]$ satisfies:
$$0\leq G[k]\leq r_{max} +(r_{min}/t_{min})t_{max} \quad \forall k \in \{1, 2, 3, \ldots\}$$

\subsection{Stochastic assumptions} \label{section:stochastic} 

Fix $(\Omega, \script{F}, P)$ as the probability space.  
The probability space contains random matrices $\{A[k]\}_{k=1}^{\infty}$ and a random variable $U$ with the following structure:

\begin{itemize} 
\item  Assume $\{A[k]\}_{k=1}^{\infty}$ is a sequence of independent and identically distributed (i.i.d.) random matrices. Each matrix $A[k]$ has size $M[k]\times (n+2)$, where $M[k]$ is a random variable that takes positive integer values. Each of the rows $r \in \{1, \ldots, M[k]\}$ has the form \eqref{eq:row-form}. Given  $M[k]=m$, the random matrix $A[k]$ has size $m\times (n+2)$ and its entries are random variables that have an arbitrary joint distribution, with the only stipulation that all rows surely satisfy the boundedness assumptions \eqref{eq:bound1}-\eqref{eq:bound4}.

\item Assume there is a random variable $U:\Omega\rightarrow [0,1]$ that is uniformly distributed over $[0,1]$ and  independent of $\{A[k]\}_{k=1}^{\infty}$. The random variable $U$ can be used, if desired, as an independent source of randomness to facilitate potentially randomized row selection decisions.\footnote{Formally, a single $U \sim \mbox{Unif}[0,1]$ can be measurably mapped to an infinite sequence $\{U[k]\}_{k=1}^{\infty}$ of i.i.d. $\mbox{Unif}[0,1]$ random variables  that, if desired, can be accessed sequentially over tasks $k \in \{1, 2, 3, \ldots\}$.}
\end{itemize}

\subsection{The sets $\Gamma$ and $\overline{\Gamma}$} 

For each task $k\in\{1, 2, 3, \ldots\}$, define a \emph{decision vector} $(T[k], R[k], Y[k])$ to be a random vector that satisfies 
$$(T[k],R[k], Y[k]) \in Row(A[k])$$
Let $\Gamma\subseteq\mathbb{R}^{n+2}$ be the set of all expectations 
$\expect{(T[k], R[k],Y[k])}$ for a given task $k$, considering all possible decision vectors. 
The set $\Gamma$ considers all conditional probabilities for choosing a row given the observed $A[k]$. The $\{A[k]\}_{k=1}^{\infty}$ matrices are i.i.d. and so $\Gamma$ is the same for all $k\in\{1, 2, 3, \ldots\}$. It can be shown that $\Gamma$ is nonempty, bounded, and convex  (see \cite{sno-text}). Its closure $\overline{\Gamma}$ is compact and convex. 

For any sequence of decision vectors and for  $m \in \{1, 2, 3, \ldots\}$, define 
\begin{equation*} 
\expect{(\overline{T}[m], \overline{R}[m], \overline{Y}[m])} = \frac{1}{m}\sum_{k=1}^m\expect{(T[k], R[k], Y[k])}
\end{equation*}  
The right-hand-side is a convex combination of points in the convex set $\Gamma$ and so
\begin{equation} \label{eq:in-gamma} 
\expect{(\overline{T}[m], \overline{R}[m], \overline{Y}[m])}  \in \Gamma \quad \forall m \in \{1, 2, 3, \ldots\}
\end{equation} 
Define the \emph{history} up to task $m$ as
$$H[m] = (A[1], A[2], \ldots, A[m-1])$$ 
where $H[1]$ is defined to be the constant $0$. 
The following lemma collects results from \cite{sno-text}\cite{neely-renewal-jmlr}.

\begin{lem} \label{lem:1} Suppose $\{A[k]\}_{k=1}^{\infty}$ are i.i.d. and satisfy the boundedness assumptions \eqref{eq:bound1}-\eqref{eq:bound4}. Then

a) For every $(t,r,y) \in \Gamma$ and $k \in \{1, 2, 3, \ldots\}$, there exists a decision vector $(T^*[k], R^*[k], Y^*[k]) \in Row(A[k])$  that is independent of $H[k]$ and 
that satisfies (with probability 1):
$$ \expect{(T^*[k], R^*[k], Y^*[k])|H[k]} = (t,r,y)$$

b)  If $\{(T[k], R[k], Y[k])\}_{k=1}^{\infty}$ is a sequence of decision vectors from a causal decision policy that, for each $k$, makes the decision $(T[k], R[k], Y[k]) \in Row(A[k])$ as a measurable function of $\{U, A[1], A[2], \ldots, A[k]\}$, then the following sample path result holds:
\begin{equation} \label{eq:sample-path-dist} 
\lim_{m\rightarrow\infty} \mbox{dist}\left((\overline{T}[m], \overline{R}[m], \overline{Y}[m])  , \overline{\Gamma}\right) = 0 \quad \mbox{(with prob 1)} 
\end{equation} 
where $\mbox{dist}(y, \Gamma)$ denotes the Euclidean distance between a vector $y \in \mathbb{R}^{n+2}$ and the convex set $\overline{\Gamma} \subseteq\mathbb{R}^{n+2}$.  
\end{lem} 

\begin{proof} 
The proof is similar to arguments in \cite{sno-text}\cite{neely-renewal-jmlr}. For completeness, the Appendix fills in minor details.  
\end{proof}

\subsection{The deterministic problem} \label{section:deterministic} 

Consider the following deterministic problem 
\begin{align}
\mbox{Maximize:} &\quad r/t \label{eq:det1}\\
\mbox{Subject to:} & \quad y_i\leq 0 \quad \forall i \in \{1, ..., n\}\label{eq:det2} \\
& \quad (t,r,y) \in \overline{\Gamma} \label{eq:det3}
\end{align}
where $\overline{\Gamma}$ is the closure of $\Gamma$.  Recall that $(t,r,y)\in\overline{\Gamma}$ implies $t\geq t_{min}>0$ and so there are no divide-by-zero issues. 
Using \eqref{eq:in-gamma} and arguments similar to those given in \cite{sno-text}, it can be shown that:  (i) 
The stochastic problem \eqref{eq:p1}-\eqref{eq:p3} is feasible if and only if the deterministic problem \eqref{eq:det1}-\eqref{eq:det3} is 
feasible; (ii) If feasible, the optimal objective values are the same \cite{sno-text}. Specifically, if $(t^*, r^*, y^*)$ solves \eqref{eq:det1}-\eqref{eq:det3} then
$$ \theta^* = r^*/t^*$$
where $\theta^*$ is the optimal objective for both the stochastic problem \eqref{eq:p1}-\eqref{eq:p3} and the deterministic problem \eqref{eq:det1}-\eqref{eq:det3}.

Assume the following \emph{Slater condition} holds: There is a value $s>0$ and a vector  
$(t^s, r^s, y^s) \in \overline{\Gamma}$ such that 
\begin{equation} \label{eq:Slater}
y_i^s \leq -s \quad \forall i \in \{1, \ldots, n\}
 \end{equation}  

\section{Algorithm} 

\subsection{Parameters and constants}

The algorithm uses parameters $v>0$, $\alpha>0$, $q=(q_1, \ldots, q_n)$ with $q_i\geq 0$ for all $i \in \{1, \ldots, n\}$, to be precisely determined later.  The constants $t_{min}, t_{max}, r_{max}, y_{i,min}, y_{i,max}, c$  from the boundedness assumptions \eqref{eq:bound1}-\eqref{eq:bound4} are assumed known. Define
\begin{align}
\gamma_{min} &= 1/t_{max} \label{eq:gamma-min} \\
\gamma_{max} &= 1/t_{min} \label{eq:gamma-max}
\end{align}
The algorithm introduces a sequence of auxiliary variables $\{\gamma[k]\}_{k=0}^{\infty}$ with initial condition $\gamma[0]=\gamma_{min}$, and with $\gamma[k]$ chosen in the interval $[\gamma_{min}, \gamma_{max}]$ for each task $k \in \{1, 2, 3, \ldots\}$. 

\subsection{Intuition} 

For intuition,  temporarily assume time averages converge to constants with probability 1. Define: 
\begin{align*}
\overline{Y}_i &= \lim_{m\rightarrow\infty}\frac{1}{m}\sum_{k=1}^mY_i[k]\\
\overline{1/\gamma} &= \lim_{m\rightarrow\infty}\frac{1}{m}\sum_{k=1}^m1/\gamma[k]
\end{align*}
The idea is to solve the following time averaged problem:
\begin{align}
\mbox{Maximize:} &\quad \lim_{m\rightarrow\infty}\frac{\frac{1}{m}\sum_{k=1}^m R[k]\gamma[k]/\gamma[k-1]}{\frac{1}{m}\sum_{k=1}^m 1/\gamma[k-1]} \label{eq:int1}\\
\mbox{Subject to:} &\quad \overline{Y}_i\leq 0 \quad \forall i \in \{1, \ldots, n\} \label{eq:int2} \\
&\quad \overline{T}\leq \overline{1/\gamma}\label{eq:int3} \\
&\quad \gamma[k] \in [\gamma_{min}, \gamma_{max}]\quad \forall k \label{eq:int4} \\
&\quad (T[k], R[k], Y[k]) \in Row(A[k]) \quad \forall k \label{eq:int5}\\
&\quad \mbox{$\gamma[k]$ varies ``slowly'' over $k$}  \label{eq:int6}
\end{align}
This is an informal description of our goals because the constraint ``$\gamma[k]$ varies slowly'' is not precise (also, the above problem assumes limits exist).  Intuitively, if $\gamma[k]$ does not change much from one task to the next, the above objective is close to $\overline{R}/\overline{1/\gamma}$, which (by the second constraint) is less than or equal to the desired objective $\overline{R}/\overline{T}$. This is useful because, as we show, the above problem can be  treated using a novel hierarchical optimization method. 

\subsection{Virtual queues} 

To enforce the constraints $\overline{Y}_i\leq 0$, for each $i \in \{1, \ldots, n\}$ define a process $Q_i[k]$ with initial condition $Q_i[1]=0$ and update equation 
\begin{equation} \label{eq:q-update}
Q_i[k+1] = [Q_i[k] + Y_i[k]]_0^{q_iv} \quad \forall k \in \{1, 2, 3, \ldots\}
\end{equation} 
where $v>0$ and $q=(q_1, \ldots, q_n)$ are given nonnegative parameters (to be precisely sized later), and 
where $[z]_0^{q_iv}$ denotes the projection of the real number $z$ onto the interval $[0, q_iv]$. Specifically
$$ [z]_0^{q_iv} = \left\{\begin{array}{cc}
q_iv & \mbox{ if $z>q_iv$} \\
z & \mbox{ if $z \in [0, q_iv]$} \\
0 & \mbox{ else} 
\end{array}\right.$$
To enforce the constraint $\overline{T}\leq \overline{1/\gamma}$, 
define a process $J[k]$ by 
\begin{equation} \label{eq:j-update} 
J[k+1] = [J[k] + T[k]-1/\gamma[k]]_0^{\infty}  \quad \forall k \in \{1, 2, 3, \ldots\}
\end{equation}
with initial condition $J[1]=0$. By construction, $Q_i[k]$ is nonnegative and upper-bounded by $q_iv$, while $J[k]$ is merely nonnegative. 
The processes $Q_i[k]$ and $J[k]$ shall be called \emph{virtual queues} because their update resembles a queueing system with arrivals and service for each $k$. Such virtual queues are standard for enforcing time average inequality constraints in stochastic systems (see \cite{sno-text}\cite{neely-energy-it}).

For each task $k$ and each $i\in\{1, \ldots, n\}$, define $1_i[k]$ as the following indicator 
function: 
\begin{equation} \label{eq:one-i} 
1_i[k]= \left\{\begin{array}{cc}
1 & \mbox{ if $Q_i[k]> q_iv-y_{i,max}$} \\
0 & \mbox{ else} 
\end{array}\right.
\end{equation}

\begin{lem}\label{lem:vq-bounds} Fix $k_0$ and $m$ as positive integers. Under the iterations \eqref{eq:q-update} and \eqref{eq:j-update}, we have for any decision vectors $\{(T[k], R[k], Y[k])\}_{k=1}^{\infty}$ 
and for all $i \in \{1, \ldots, n\}$: 
 \begin{align}
&\frac{1}{m} \sum_{k=k_0}^{k_0+m-1} (Y_i[k]-1_i[k]y_{i,max}) \leq \frac{q_iv}{m}\label{eq:vq1} \\
&\frac{1}{m}\sum_{k=k_0}^{k_0+m-1}(T[k]-1/\gamma[k]) \leq \frac{J[k_0+m]}{m}\label{eq:vq2}
 \end{align}
\end{lem} 
\begin{proof} 
Fix $i \in \{1, \ldots, n\}$ and $k \in \{1, 2, 3, \ldots\}$.  We first claim
\begin{equation} \label{eq:first-claim} 
Y_i[k] - 1_i[k]y_{i,max} \leq Q_i[k+1] - Q_i[k] 
\end{equation} 
To verify \eqref{eq:first-claim}, consider the two cases: 
\begin{itemize} 
\item Case 1: Suppose $Q_i[k]+Y_i[k]>q_iv$.  It follows by \eqref{eq:q-update} that $Q_i[k+1]=q_iv$. Since $Y_i[k]\leq y_{i,max}$, we have $Q_i[k]>q_iv-y_{i,max}$ and so $1_i[k]=1$.  It follows that \eqref{eq:first-claim} reduces to $Y_i[k]-y_{i,max}\leq q_iv -Q_i[k]$, which is true because the left-hand-side is always nonpositive while the right-hand-side is always nonnegative. 
\item Case 2: Suppose $Q_i[k]+Y_i[k]\leq q_iv$. The update \eqref{eq:q-update} then gives $Q_i[k+1]\geq Q_i[k]+Y_i[k]$, so \eqref{eq:first-claim} again holds (recall $y_{i,max}\geq 0$). 
\end{itemize} 
Summing \eqref{eq:first-claim} over $k \in \{k_0, \ldots, k_0+m-1\}$ gives
\begin{align*} 
\sum_{k=k_0}^{k_0+m-1}(Y_i[k]-1_i[k]y_{i,max}) 
&\leq Q_i[k_0+m]-Q_i[k_0]\\
&\leq q_iv
\end{align*}
where the final equality holds because the update  \eqref{eq:q-update} ensures $Q_i[k] \in [0, q_iv]$ for all $k$. 
Dividing by $m$ proves \eqref{eq:vq1}. 

To prove \eqref{eq:vq2}, observe the update \eqref{eq:j-update} implies 
\begin{equation} \label{eq:jt-bound}
J[k+1] \geq J[k] + T[k]-1/\gamma[k] \quad \forall k \in \{1, 2, 3, \ldots\}
\end{equation} 
Summing over $k \in \{k_0, \ldots, k_0+m-1\}$ gives 
$$ J[k_0+m]-J[k_0] \geq \sum_{k=k_0}^{k_0+m-1}(T[k]-1/\gamma[k])$$
The result follows by dividing by $m$ and observing $J[k_0]\geq 0$. 
\end{proof}

The lemma shows the following: To ensure the desired time average inequality constraints are close to being satisfied over a sequence of $m$ consecutive tasks, decisions should be made to keep $J[k]$ bounded (so the right-hand-side of \eqref{eq:vq2} vanishes as $m$ gets large) and to ensure $Y_i[k]$ rarely crosses the $q_iv - y_{i,max}$ threshold (so  $1_i[k]$ on the left-hand-side of \eqref{eq:vq1} is rarely nonzero).

\subsection{Lyapunov drift} 

Define $Q[k]=(Q_1[k], \ldots, Q_n[k])$. 
Define 
$$L[k]=\frac{1}{2}J[k]^2 + \frac{1}{2}\norm{Q[k]}^2$$
where $\norm{Q[k]}^2 = \sum_{i=1}^nQ_i[k]^2$. The process $L[k]$ can be viewed as a Lyapunov function on the queue state for task $k$. Define 
$$\Delta[k]=L[k+1]-L[k]$$

\begin{lem} Under the iterations \eqref{eq:q-update} and \eqref{eq:j-update}, we have for any decision vectors 
$\{(T[k], R[k], Y[k])\}_{k=1}^{\infty}$ and for all positive integers $k$
\begin{equation} \label{eq:Delta}
\Delta[k] \leq b + J[k](T[k]-1/\gamma[k]) + Q[k]^{\top}Y[k]
 \end{equation}
 where $b$ is a constant defined by 
$$ b = \frac{1}{2}\left[c^2 +(t_{max}-t_{min})^2\right] $$
with $c$ given in \eqref{eq:bound3}, and 
where $Q[k]^{\top}Y[k]=\sum_{i=1}^kQ_i[k]Y_i[k]$.
\end{lem} 

\begin{proof} 
Fix $k \in \{1, 2, 3, \ldots\}$. Squaring \eqref{eq:q-update} and using $([z]_0^{q_iv})^2\leq z^2$ for all $z \in \mathbb{R}$ gives 
\begin{align*}
Q_i[k+1]^2 &\leq (Q_i[k]+Y_i[k])^2\\
&=Q_i[k]^2 + Y_i[k]^2 + 2Q_i[k]Y_i[k]
\end{align*}
Summing over $i \in \{1, \ldots, n\}$ and dividing by 2 gives 
\begin{align} 
\frac{1}{2}\norm{Q[k+1]}^2&\leq \frac{1}{2}\norm{Q[k]}^2 + \frac{1}{2}\norm{Y[k]}^2 + Q[k]^{\top}Y[k]\nonumber\\
&\leq \frac{1}{2}\norm{Q[k]}^2 + \frac{1}{2}c^2 + Q[k]^{\top}Y[k] \label{eq:LDa}
\end{align}
where we have used the boundedness assumption \eqref{eq:bound3}.
Similarly, squaring \eqref{eq:j-update} and using $([z]_0^{\infty})^2\leq z^2$ for all $z \in \mathbb{R}$ gives  
\begin{align}
\frac{1}{2}J[k+1]^2 \leq \frac{1}{2}J[k]^2 + \frac{1}{2}(T[k]-1/\gamma[k])^2 + J[k](T[k]-1/\gamma[k])\label{eq:LDb}
\end{align}
Summing \eqref{eq:LDa} and \eqref{eq:LDb} gives
$$ \Delta[k] \leq \frac{1}{2}[c^2 + (T[k]-1/\gamma[k])^2] + J[k](T[k]-1/\gamma[k])+Q[k]^{\top}Y[k]$$
The result follows by observing that 
$$(T[k]-1/\gamma[k])^2 \leq (t_{max}-t_{min})^2$$
which holds by the boundedness assumption \eqref{eq:bound1} and the fact $1/\gamma[k]\in [t_{min}, t_{max}]$.
\end{proof} 

\subsection{Discussion} 

The above lemma implies that 
\begin{equation} \label{eq:DPP-explain} 
\Delta[k] - vR[k] \leq b -vR[k] +  J[k](T[k]-1/\gamma[k]) + Q[k]^{\top}Y[k]
\end{equation} 
The expression $\Delta[k]-vR[k]$ is called the \emph{drift plus penalty expression} because $\Delta[k]$ is associated with the drift of the Lyapunov function $L[k]$, and $-vR[k]$ is a weighted version of the reward for task $k$ (multiplied by $-1$ to turn a reward into a penalty). 
As shown in \cite{sno-text}, algorithms that minimize the right-hand-side of similar drift-plus-penalty expressions can treat stochastic problems that seek to minimize the time average of an objective function subject to time average inequality constraints.  This could be used to treat the problem 
\eqref{eq:int1}-\eqref{eq:int6} if the objective \eqref{eq:int1} were changed to minimizing $-\overline{R}$ and if the (ambiguous) constraint \eqref{eq:int6} were removed. 

We cannot use the drift-plus-penalty method for problem \eqref{eq:int1}-\eqref{eq:int6} because the objective is a ratio of averages, rather than a single average. For this, we design a novel hierarchical version of the drift-plus-penalty method that, for each task $k$ does: 
\begin{itemize} 
\item Step 1: Choose $(T[k], R[k], Y[k]) \in Row(A[k])$ to greedily minimize the right-hand-side of \eqref{eq:DPP-explain} (ignoring the term of this right-hand-side that depends on $\gamma[k]$); 
\item Step 2: Treating $\gamma[k-1], T[k], R[k], Y[k]$ as known constants, choose $\gamma[k] \in [\gamma_{min}, \gamma_{max}]$ to minimize 
$$ \underbrace{\frac{\gamma[k]}{\gamma[k-1]}(-vR[k] + J[k](T[k]-1/\gamma[k])+ Q[k]^{\top}Y[k])}_{\mbox{for \eqref{eq:int1}}} + \underbrace{\frac{\alpha v^2}{2}(\gamma[k]-\gamma[k-1])^2}_{\mbox{for \eqref{eq:int6}}}$$
\end{itemize} 
where the term in the first underbrace relates to the desired objective \eqref{eq:int1} and arises by multiplying both sides of \eqref{eq:DPP-explain} by $\gamma[k]/\gamma[k-1]$; the term in the second underbrace is a weighted ``prox-type'' term that, for our purposes,  acts only to enforce constraint \eqref{eq:int6}.

\subsection{Algorithm} 

Fix parameters $v>0, \alpha>0, q=(q_1, \ldots, q_n)$ with $q_i\geq 0$ for $i \in \{1, \ldots, n\}$ (to be sized later). 
Fix $\gamma[0]=\gamma_{min}$. For each task $k\in\{1, 2, 3, \ldots\}$, the algorithm proceeds as follows: 

\begin{itemize} 
\item Row selection: Observe $Q[k], J[k], A[k]$ and treat these as given constants. Choose $(T[k], R[k], Y[k]) \in Row(A[k])$ to minimize
$$ -vR[k] + J[k]T[k] + Q[k]^{\top}Y[k]$$
In the case of ties, break the tie in favor of the smallest indexed row. 
\item $\gamma[k]$ selection: Observe $Q[k], J[k], \gamma[k-1]$, and the decisions $(T[k], R[k], Y[k])$ just made by the row selection, and treat these as given constants. Choose $\gamma[k]\in [\gamma_{min}, \gamma_{max}]$ to minimize
$$ \gamma[k]\left[-vR[k] + J[k]T[k]+Q[k]^{\top}Y[k]\right] + \frac{\gamma[k-1]\alpha v^2}{2}(\gamma[k]-\gamma[k-1])^2$$
The explicit solution to this quadratic minimization is: 
\begin{equation} \label{eq:gamma}
\gamma[k]= \left[\gamma[k-1] + \frac{vR[k]-J[k]T[k]-Q[k]^{\top}Y[k]}{\gamma[k-1]\alpha v^2}\right]_{\gamma_{min}}^{\gamma_{max}}
\end{equation} 
where$[z]_{\gamma_{min}}^{\gamma_{max}}$ denotes the projection of $z\in\mathbb{R}$ onto the interval $[\gamma_{min}, \gamma_{max}]$. 
\item Virtual queue updates: Update the virtual queues via \eqref{eq:q-update} and \eqref{eq:j-update}.
\end{itemize} 

\subsection{Basic analysis} 

Fix $k \in \{1, 2, 3, \ldots\}$. The row selection decision of our algorithm implies 
\begin{equation} \label{eq:analysis1} 
-vR[k]+J[k]T[k]+Q[k]^{\top}Y[k]\leq -vR^*[k]+J[k]T^*[k]+Q[k]^{\top}Y^*[k]
\end{equation} 
where $(T^*[k], R^*[k], Y^*[k])$ is any other vector in $Row(A[k])$ (including any decision vector for task $k$ that is chosen according to some optimized probability distribution).

The $\gamma[k]$ selection decision of our algorithm chooses $\gamma[k] \in [\gamma_{min}, \gamma_{max}]$ to minimize a function of $\gamma[k]$
that is $\beta$-strongly convex for parameter $\beta=\gamma[k-1]\alpha v^2$.  Therefore, by standard \emph{strongly convex pushback} results (see, for example, Lemma 2.1 in \cite{SGD-robust}, \cite{tseng-accelerated}, Lemma 6 in \cite{neely-convex-chapter}), the following holds for any other $\gamma^*\in [\gamma_{min}, \gamma_{max}]$: 
\begin{align*}
&\gamma[k]\left[-vR[k] + J[k]T[k]+Q[k]^{\top}Y[k]\right] + \frac{\gamma[k-1]\alpha v^2}{2}(\gamma[k]-\gamma[k-1])^2\\
&\leq \gamma^*\left[-vR[k] + J[k]T[k]+Q[k]^{\top}Y[k]\right] + \frac{\gamma[k-1]\alpha v^2}{2}(\gamma^*-\gamma[k-1])^2- \underbrace{\frac{\gamma[k-1]\alpha v^2}{2}(\gamma^*-\gamma[k])^2}\\
&\leq \gamma^*\left[-vR^*[k] + J[k]T^*[k]+Q[k]^{\top}Y[k]^*\right] +  \frac{\gamma[k-1]\alpha v^2}{2}(\gamma^*-\gamma[k-1])^2- \frac{\gamma[k-1]\alpha v^2}{2}(\gamma^*-\gamma[k])^2
\end{align*}
where the first inequality gives an underbrace to highlight the pushback term that arises from strong convexity; the second inequality holds by  \eqref{eq:analysis1} and the fact $\gamma^*>0$.  

Dividing the above inequality by $\gamma[k-1]>0$ gives
\begin{align}
&\frac{\gamma[k]}{\gamma[k-1]}[-vR[k]+J[k]T[k]+Q[k]^{\top}Y[k]] + \frac{\alpha v^2}{2}(\gamma[k]-\gamma[k-1])^2 \nonumber\\
&\leq \frac{\gamma^*}{\gamma[k-1]}[-vR^*[k]+J[k]T^*[k]+Q[k]^{\top}Y^*[k]] +  \frac{\alpha v^2}{2}(\gamma^*-\gamma[k-1])^2 - \frac{\alpha v^2}{2}(\gamma^*-\gamma[k])^2 \label{eq:major} 
\end{align}

The following lemma is obtained by mere arithmetic rearrangements of the inequality \eqref{eq:major}.

\

\begin{lem} For any sample path of $\{A[k]\}_{k=1}^{\infty}$, for each $k \in \{1, 2, 3, \ldots\}$ 
our algorithm yields a drift-plus-penalty expression $\Delta[k]-vR[k]$ that surely satisfies 
\begin{align}
\Delta[k] - vR[k] &\leq b + \frac{\gamma^*}{\gamma[k-1]}[-vR^*[k]+J[k](T^*[k]-1/\gamma^*) + Q[k]^{\top}Y^*[k]]\nonumber\\
&\quad  + \frac{\alpha v^2}{2}(\gamma^*-\gamma[k-1])^2 - \frac{\alpha v^2}{2}(\gamma^*-\gamma[k])^2\nonumber\\
&\quad + \frac{(vr_{max} + cv\norm{q} + (t_{max}-t_{min})|J[k]|)^2}{2\gamma_{min}^2\alpha v^2} \label{eq:goalRHS3}
\end{align}
where $\Delta[k]-vR[k], \gamma[k], \gamma[k-1]$ refer to the actual values that arise in our algorithm;  
$(T^*[k], R^*[k], Y^*[k]) \in Row(A[k])$ is \emph{any} decision vector for task $k$ (not necessarily the decision vector chosen by our algorithm); $\gamma^*$ is \emph{any} real number in the interval $[\gamma_{min}, \gamma_{max}]$.
\end{lem} 
\begin{proof} 
Adding $-J[k]/\gamma[k-1]$ to both sides of \eqref{eq:major} gives  
\begin{align}
LHS[k] &:=\frac{\gamma[k]}{\gamma[k-1]}[-vR[k]+J[k](T[k]-1/\gamma[k])+Q[k]^{\top}Y[k]] + \frac{\alpha v^2}{2}(\gamma[k]-\gamma[k-1])^2 \nonumber \\
&\leq \frac{\gamma^*}{\gamma[k-1]}[-vR^*[k]+J[k](T^*[k]-1/\gamma^*)+Q[k]^{\top}Y^*[k]] +  \frac{\alpha v^2}{2}(\gamma^*-\gamma[k-1])^2 - \frac{\alpha v^2}{2}(\gamma^*-\gamma[k])^2 \label{eq:parta} 
\end{align}
where, for simplicity of the arithmetic, we have defined $LHS[k]$ according to the \emph{Left-Hand-Side} of the above inequality.  
By rearranging terms in the definition of $LHS[k]$ we have 
\begin{align} 
LHS[k] &= -vR[k] + J[k](T[k]-1/\gamma[k]) + Q[k]^{\top}Y[k]\nonumber \\
&\quad+ \frac{(\gamma[k]-\gamma[k-1])}{\gamma[k-1]}[-vR[k]+J[k](T[k]-1/\gamma[k])+Q[k]^{\top}Y[k]]\nonumber\\ 
&\quad +  \frac{\alpha v^2}{2}(\gamma[k]-\gamma[k-1])^2\nonumber\\
&\geq \Delta[k] -vR[k] - b   + \frac{(\gamma[k]-\gamma[k-1])}{\gamma[k-1]}[-vR[k]+J[k](T[k]-1/\gamma[k])+Q[k]^{\top}Y[k]]\nonumber\\ 
&\quad +  \frac{\alpha v^2}{2}(\gamma[k]-\gamma[k-1])^2\nonumber\\
&\geq \Delta[k] - vR[k]-b - \frac{\left(\frac{-vR[k]+J[k](T[k]-1/\gamma[k])+Q[k]^{\top}Y[k]}{\gamma[k-1]} \right)^2}{2\alpha v^2} \label{eq:bazx} 
\end{align} 
where the first inequality holds by \eqref{eq:Delta}; the final inequality holds by the fact that for all real numbers $x, y$: 
$$ yx + \frac{\alpha v^2}{2}x^2 \geq -\frac{y^2}{2\alpha v^2}$$
which holds by completing the square (in this case we use $x=\gamma[k]-\gamma[k-1]$).  Now observe that 
\begin{align*}
&|-vR[k]| \leq vr_{max}\\
&|J[k](T[k]-1/\gamma[k])| \leq (t_{max}-t_{min})|J[k]|\\
&|Q[k]^{\top}Y[k]| \leq \norm{Q[k]}\cdot \norm{Y[k]} \leq cv\norm{q}
\end{align*}
where the final inequality uses Cauchy-Schwarz, the boundedness assumption \eqref{eq:bound3}, and the fact $0\leq Q_i[k]\leq q_iv$ for all $i$. Substituting these bounds into the right-hand-side of \eqref{eq:bazx} gives 
$$ LHS[k]\geq \Delta[k] - vR[k]-b - \frac{(vr_{max} +cv \norm{q} + (t_{max}-t_{min})|J[k]|)^2}{2\gamma[k-1]^2\alpha v^2}$$
Substituting this into \eqref{eq:parta} proves the result. 
\end{proof}

\subsection{Expected drift-plus-penalty} 

Recall that $H[k]=(A[1], \ldots, A[k-1])$ and note that $H[k]$ determines $Q[k]$ and $J[k]$, that is, $(Q[k],J[k])$ is $H[k]$-measurable.

\begin{lem} Suppose the problem \eqref{eq:det1}-\eqref{eq:det3} is feasible with optimal solution $(t^*, r^*, y^*) \in \overline{\Gamma}$ and optimal objective value $\theta^*=r^*/t^*$. Then for 
all $k\in\{1, 2, 3, \ldots\}$ we have (with probability 1): 
\begin{align}
\expect{\Delta[k]-vR[k]|H[k]} &\leq b + \frac{-v\theta^*}{\gamma[k-1]}+\frac{\alpha v^2}{2}\expect{(1/t^*-\gamma[k-1])^2 - (1/t^*-\gamma[k])^2|H[k]}\nonumber\\
&\quad + \frac{(vr_{max} + cv\norm{q} + (t_{max}-t_{min})|J[k]|)^2}{2\gamma_{min}^2\alpha v^2} \label{eq:DPP}
\end{align}
\end{lem} 

\begin{proof} 
Fix $k \in \{1, 2, 3, \ldots\}$. Fix $(t,r,y)\in \Gamma$.  Observe that $t \in [t_{min}, t_{max}]$ and so $1/t \in [\gamma_{min}, \gamma_{max}]$. By Lemma \ref{lem:1}a, there is a decision vector $(T^*[k], R^*[k], Y^*[k]) \in Row(A[k])$ that is independent of $H[k]$ such that 
\begin{equation} \label{eq:cond-exp} 
\expect{(T^*[k], R^*[k], Y^*[k])|H[k]} = (t, r, y)
\end{equation} 
Substituting this $(T^*[k], R^*[k], Y^*[k])$, along with $\gamma^*=1/t$, into \eqref{eq:goalRHS3} gives 
\begin{align}
\Delta[k] - vR[k] &\leq b + \frac{1/t}{\gamma[k-1]}[-vR^*[k]+J[k](T^*[k]-t) + Q[k]^{\top}Y^*[k]]\nonumber\\
&\quad  + \frac{\alpha v^2}{2}(1/t-\gamma[k-1])^2 - \frac{\alpha v^2}{2}(1/t-\gamma[k])^2\nonumber\\
&\quad + \frac{(vr_{max} + cv\norm{q} + (t_{max}-t_{min})|J[k]|)^2}{2\gamma_{min}^2\alpha v^2} \label{eq:goalRHS}
\end{align}
Taking conditional expectations and using \eqref{eq:cond-exp} gives 
\begin{align}
\expect{\Delta[k] - vR[k]|H[k]} &\leq b + \frac{1/t}{\gamma[k-1]}[-vr + Q[k]^{\top}y]\nonumber\\
&\quad  + \frac{\alpha v^2}{2}\expect{(1/t-\gamma[k-1])^2 - (1/t-\gamma[k])^2|H[k]}\nonumber\\
&\quad + \frac{(vr_{max} + cv\norm{q} + (t_{max}-t_{min})|J[k]|)^2}{2\gamma_{min}^2\alpha v^2} \label{eq:goalRHS2}
\end{align}
This holds for all $(t,r,y)\in\Gamma$. Since $(t^*, r^*, y^*)\in\overline{\Gamma}$, there is a sequence of points $\{(t_j, r_j, y_j)\}_{j=1}^{\infty}$ in $\Gamma$ that converges to $(t^*, r^*, y^*)$. Taking a limit over such points in \eqref{eq:goalRHS2} gives 
\begin{align*}
\expect{\Delta[k] - vR[k]|H[k]} &\leq b + \frac{1/t^*}{\gamma[k-1]}[-vr^* + Q[k]^{\top}y^*]\nonumber\\
&\quad  + \frac{\alpha v^2}{2}\expect{(1/t^*-\gamma[k-1])^2 - (1/t^*-\gamma[k])^2|H[k]}\nonumber\\
&\quad + \frac{(vr_{max} + cv\norm{q} + (t_{max}-t_{min})|J[k]|)^2}{2\gamma_{min}^2\alpha v^2} 
\end{align*}
Recall the optimal solution has $y^*=(y_1^*, \ldots, y_n^*)$ with $y_i^*\leq 0$ for all $i$. 
The result is obtained by substituting $Q[k]^{\top}y^*\leq 0$ and 
$r^*/t^*=\theta^*$.
\end{proof}

\section{Reward guarantee}

Fix $\epsilon>0$. This section proves that our algorithm yields
$$\frac{\expect{\overline{R}[m]}}{\expect{\overline{T}[m]}} \geq \theta^* - O(\epsilon)$$
when $m$ is suitably large and when parameters $v, q$ are sized appropriately with $\epsilon$. This shows that desirable reward performance holds over the tasks $\{1, \ldots, m\}$. 
More strongly, this section shows the same result holds over any sequence of $m$ consecutive tasks. The corresponding result for the time average penalty constraints $\overline{Y}[m]$ is shown in Section \ref{section:constraints}.

\subsection{Deterministic bound on $J[k]$} 

\begin{lem} Under any $\{A[k]\}_{k=1}^{\infty}$ sequence,  our algorithm 
yields 
\begin{equation} \label{eq:J-bound}
 0\leq J[k]\leq v(\beta_1+\beta_2) 
 \end{equation} 
where $\beta_1$ and $\beta_2$ are nonnegative constants defined 
\begin{align}
\beta_1  &= \frac{(1+r_{max} + \sum_{i=1}^n q_iy_{i,min})}{t_{min}} \label{eq:theta1}\\
\beta_2 &=  \frac{1}{v}\left\lceil\alpha v\gamma_{max}(\gamma_{max} - \gamma_{min})\right\rceil (t_{max} - t_{min})\label{eq:theta2}
\end{align}
where $\lceil z \rceil$ denotes the smallest integer greater than or equal to the real number $z$.
\end{lem} 
\begin{proof} 
Define 
\begin{equation} \label{eq:m}
m = \left\lceil\alpha v\gamma_{max}(\gamma_{max} - \gamma_{min})\right\rceil
\end{equation} 
We first make two claims: 
\begin{itemize} 
\item Claim 1: If $J[k_0]\leq v\beta_1$ for some task $k_0 \in \{1, 2, 3, \ldots\}$, then 
$$J[k]\leq v(\beta_1+\beta_2) \quad \forall k \in \{k_0, k_0+1, \ldots, k_0+m\}$$
To prove Claim 1, observe that for each task $k$ we have 
$$ T[k]-1/\gamma[k] \leq t_{max} - 1/\gamma_{max} = t_{max} - t_{min}$$
Thus, the update \eqref{eq:j-update} implies that $J[k]$ can increase by at most $t_{max}-t_{min}$ 
on any given task $k$.  Thus, $J[k]$ can increase by at most $m(t_{max}-t_{min})$ 
over any sequence of $m$ or fewer tasks. By construction, $m(t_{max}-t_{min})=v\beta_2$. 
It follows that if $J[k_0]\leq v\beta_1$ then $J[k]\leq v\beta_1 + v\beta_2$ for all $k \in \{k_0, k_0+1, \ldots, k_0+m\}$.
\item Claim 2: If $J[k]\geq v\beta_1$ for some task $k \in \{1, 2, 3, \ldots\}$ then $\gamma[k]\leq \gamma[k-1]$, and in particular
$$  \gamma[k]\leq \max\left\{\gamma_{min}, \gamma[k-1]- \frac{1}{\gamma_{max} \alpha v}\right\}$$
To prove Claim 2, suppose $J[k]\geq v\beta_1$. Observe that
\begin{align*}
\frac{vR[k]-J[k]T[k]-Q[k]^{\top}Y[k]}{\gamma[k-1]\alpha v^2}&\overset{(a)}{\leq} \frac{vr_{max} - v\beta_1t_{min} + \sum_{i=1}^nq_ivy_{i,min}}{\gamma[k-1]\alpha v^2} \\
&\overset{(b)}{=} \frac{-1}{\gamma[k-1]\alpha v}\\
&\leq \frac{-1}{\gamma_{max}\alpha v}
\end{align*}
where inequality (a) holds because $R[k]\leq r_{max}$, $T[k]\geq t_{min}$, and $-Q_i[k]Y_i[k]\leq q_ivy_{i,min}$ for all $i \in \{1, \ldots, n\}$; equality (b) holds by definition of $\beta_1$. 
Claim 2 follows in view of the iteration \eqref{eq:gamma}.
\end{itemize}

Since $J[1]=0\leq v\beta_1$, Claim 1 implies 
$J[k]\leq v(\beta_1+\beta_2)$ for all $k \in \{1, \ldots, 1+m\}$. We now use induction: Suppose $J[k]\leq v(\beta_1+\beta_2)$ for all $k \in \{1, \ldots, k_0\}$ for some positive integer $k_0\geq 1+m$. We show this is also true for $k_0+1$. If $J[k]\leq v\beta_1$ for some $k \in \{k_0+1-m, \ldots, k_0\}$ then Claim 1 implies $J[k_0+1]\leq v(\beta_1+\beta_2)$ and we are done. 

Now suppose $J[k]>v\beta_1$ for all $k \in \{k_0+1-m, \ldots, k_0\}$. Claim 2 implies 
$$  \gamma[k_0+1-m]\geq \gamma[k_0-m]\geq \ldots \geq \gamma[k_0]$$
Therefore, if $\gamma[k]=\gamma_{min}$ for some $k \in \{k_0+1-m, \ldots, k_0\}$ then 
$\gamma[k_0]=\gamma_{min}=1/t_{max}$ and the update \eqref{eq:j-update} gives 
$$ J[k_0+1]= [J[k_0] + T[k_0]-t_{max}]_0^{\infty} \leq J[k_0]\leq v(\beta_1+\beta_2)$$
and we are done. We now show the remaining case $\gamma[k]>\gamma_{min}$ for all $k \in \{k_0+1-m, \ldots, k_0\}$ is impossible. Suppose $\gamma[k]>\gamma_{min}$ for all $k \in \{k_0+1-m, \ldots, k_0\}$ (we reach a contradiction). Then Claim 2 implies 
$$ \gamma[k] -\gamma[k-1] \leq  - \frac{1}{\gamma_{max} \alpha v} \quad \forall k \in \{k_0+1-m, \ldots, k_0\}$$
Summing over $k \in \{k_0+1-m, \ldots, k_0\}$ gives 
$$\gamma[k_0]-\gamma[k_0-m] \leq -\frac{m}{\gamma_{max}\alpha v}$$ 
and so 
\begin{align*}
\gamma[k_0] &\leq \gamma_{max} - \frac{m}{\gamma_{max}\alpha v}\\
&\overset{(a)}{\leq}\gamma_{max} - (\gamma_{max}-\gamma_{min}) \\
&= \gamma_{min}
\end{align*}
where inequality (a) holds by definition of $m$ in \eqref{eq:m}. This contradicts the fact that $\gamma[k_0]>\gamma_{min}$. 
\end{proof} 

\subsection{Reward over any consecutive $m$ tasks} 

For postive integers $m, k_0$, 
define $\overline{R}[k_0;m]$ and $\overline{T}[k_0, m]$ as empirical averages over the  $m$ consecutive tasks that start with task $k_0$: 
\begin{align*}
\overline{R}[k_0;m] &= \frac{1}{m}\sum_{k=k_0}^{k_0+m-1}R[k]\\
\overline{T}[k_0;m] &= \frac{1}{m}\sum_{k=k_0}^{k_0+m-1}T[k]
\end{align*}

\begin{thm} \label{thm:reward} Suppose the problem \eqref{eq:det1}-\eqref{eq:det3} is feasible with optimal solution $(t^*, r^*, y^*) \in \overline{\Gamma}$ and optimal objective value $\theta^*=r^*/t^*$. Then for any parameters $v>0, \alpha>0$, $q=(q_1, \ldots, q_n)\geq 0$ and 
all positive integers $k_0, m$, our algorithm yields
\begin{align}
\frac{\expect{\overline{R}[k_0;m]}}{\expect{\overline{T}[k_0;m]}} &\geq \theta^* - \frac{d_1}{v} - \frac{vd_2}{m} - \frac{r_{max}/t_{min}}{m} \label{eq:reward} 
\end{align}
where $d_1, d_2$ are defined
\begin{align}
d_1 &= \frac{b + \frac{1}{2\gamma_{min}^2\alpha}(r_{max} + c\norm{q} + (t_{max} - t_{min})(\beta_1+\beta_2))^2}{t_{min}}\label{eq:d1} \\
d_2 &= \frac{\frac{1}{2}\norm{q}^2 + \frac{1}{2}(\beta_1+\beta_2)^2+ \frac{\alpha}{2}(\gamma_{max} - \gamma_{min})^2+\theta^*(\beta_1+\beta_2)}{t_{min}} \label{eq:d2}
\end{align}
In particular, fixing $q_i\geq 0$ and $\epsilon>0$ and choosing $v=1/\epsilon$, $\alpha=1$, gives for all $k_0$: 
\begin{align}
\frac{\expect{\overline{R}[k_0;m]}}{\expect{\overline{T}[k_0;m]}} &\geq \theta^* - O(\epsilon) \quad \forall m\geq 1/\epsilon^2
\end{align}
Similar behavior holds when replacing $\alpha=1$ with $\alpha = c_1/\max[c_2,1/2]$ where $c_1,c_2$ are fine tuned constants (defined later) in \eqref{eq:c1},\eqref{eq:c2}.
\end{thm} 
\begin{proof} 
Fix $k \in \{2, 3, 4, \ldots\}$. Using iterated expectations and substituting $|J[k]|\leq v(\beta_1+\beta_2)$ into \eqref{eq:DPP} gives 
\begin{align}
\expect{\Delta[k]-vR[k]} &\leq b  -v\theta^*\expect{\frac{1}{\gamma[k-1]}}+\frac{\alpha v^2}{2}\expect{(1/t^*-\gamma[k-1])^2 - (1/t^*-\gamma[k])^2}\nonumber\\
&\quad + \frac{(r_{max} + c\norm{q} + (t_{max}-t_{min})(\beta_1+\beta_2))^2}{2\gamma_{min}^2\alpha}\label{eq:R1}
\end{align}
Manipulating the second term on the right-hand-side above gives
\begin{align*}
-v\theta^*\expect{\frac{1}{\gamma[k-1]}} &= -v\theta^*\expect{T[k-1]} + v\theta^*\expect{T[k-1]-\frac{1}{\gamma[k-1]}}\\
&\leq -v\theta^*\expect{T[k-1]} + v\theta^*\expect{J[k]-J[k-1]}
\end{align*}
where the final inequality holds by \eqref{eq:jt-bound}. Substituting this into the right-hand-side of \eqref{eq:R1} gives
\begin{align}
\expect{\Delta[k]-vR[k]} &\leq b -v\theta^*\expect{T[k-1]} + v\theta^*\expect{J[k]-J[k-1]}\nonumber \\
&\quad +\frac{\alpha v^2}{2}\expect{(1/t^*-\gamma[k-1])^2 - (1/t^*-\gamma[k])^2}\nonumber\\
&\quad + \frac{(r_{max} + c\norm{q} + (t_{max}-t_{min})(\beta_1+\beta_2))^2}{2\gamma_{min}^2\alpha}\label{eq:R2}
\end{align}
Summing the above over $k \in \{k_0+1, \ldots, k_0+m\}$ and dividing by $mv$ gives
\begin{align}
&\frac{1}{mv}\expect{L[k_0+m+1]-L[k_0+1]}-\frac{1}{m}\sum_{k=k_0+1}^{k_0+m}\expect{R[k]} \nonumber\\
&\leq \frac{b}{v} -\theta^*\frac{1}{m}\sum_{k=k_0+1}^{k_0+m}\expect{T[k-1]} + \frac{\theta^*}{m}\expect{J[k_0+m]-J[k_0]}\nonumber \\
&\quad +\frac{\alpha v}{2m}\expect{(1/t^*-\gamma[k_0])^2 - (1/t^*-\gamma[k_0+m])^2}\nonumber\\
&\quad + \frac{(r_{max} + c\norm{q} + (t_{max}-t_{min})(\beta_1+\beta_2))^2}{2v\gamma_{min}^2\alpha}\nonumber\\
&\leq \frac{b}{v} -\theta^*\expect{\overline{T}[k_0;m]} + \frac{\theta^*}{m}v(\beta_1+\beta_2)\nonumber \\
&\quad +\frac{\alpha v}{2m}(\gamma_{max}-\gamma_{min})^2\nonumber\\
&\quad + \frac{(r_{max} + c\norm{q} + (t_{max}-t_{min})(\beta_1+\beta_2))^2}{2v\gamma_{min}^2\alpha}\label{eq:foofinal} 
\end{align}
where the final inequality substitutes the definition of $\overline{T}[k_0;m]$ and uses:
\begin{align*}
&J[k]\leq v(\beta_1+\beta_2) \quad \forall k \\
&(1/t^*-\gamma[k])^2 \leq (\gamma_{max}-\gamma_{min})^2 \quad \forall k 
\end{align*}
Rearranging terms in \eqref{eq:foofinal} gives
\begin{align*}
\expect{\overline{R}[k_0;m]} &\geq \theta^*\expect{\overline{T}[k_0;m]} -   \frac{\expect{R[k_0+m]-R[k_0]}}{m} \\
&\quad -\frac{1}{mv}\expect{L[k_0+1]-L[k_0+m+1]} \\
&\quad - \frac{b + \frac{1}{2\gamma_{min}^2\alpha}(r_{max} + c\norm{q} + (t_{max}-t_{min})(\beta_1+\beta_2))^2}{v} \\
&\quad - \frac{v}{m}(\frac{\alpha}{2}(\gamma_{max}-\gamma_{min})^2 + \theta^*(\beta_1+\beta_2))
\end{align*}
Terms on the right-hand-side have the following bounds: 
\begin{align*}
&-\expect{R[k_0+m]-R[k_0]} \geq -r_{max} \\
&-\expect{L[k_0+1]-L[k_0+m+1]} \geq -\expect{L[k_0+1]} \geq -\frac{1}{2}\norm{vq}^2 - \frac{1}{2}v^2(\beta_1+\beta_2)^2
\end{align*}
where the final inequality uses
$$ L[k_0+1]=\frac{1}{2}\norm{Q[k_0+1]}^2+\frac{1}{2}J[k_0+1]^2 $$
and uses the fact that for all $k$ we have $\norm{Q[k]}\leq \norm{vq}$ (since $0\leq Q_i[k]\leq q_iv$ from \eqref{eq:q-update}) and 
  $0\leq J[k]\leq v(\beta_1+\beta_2)$ (from \eqref{eq:J-bound}). 
This proves the result upon usage of the constants $d_1, d_2$. 
\end{proof}

\subsection{No penalty constraints} \label{section:n=0}

A special and nontrivial case of Theorem \ref{thm:reward} is when the only goal is to maximize $\overline{R}/\overline{T}$,  
the time average reward per unit time, with no additional constraints on the $Y_i[k]$ processes. This can be viewed as the case $n=0$ (so there are no $Y_i[k]$ processes). Equivalently, it can be treated by using $q=0$ in Theorem \ref{thm:reward}. 

To consider this $n=0$ case, let $\overline{R}[m]$ and $\overline{T}[m]$ be averages over tasks $\{1, \ldots, m\}$. Fix $\epsilon>0$.   
The work \cite{neely-renewal-jmlr} showed that, in the absence of a-priori knowledge of the probability distribution of the $A[k]$ matrices, any algorithm that runs over over tasks $\{1, 2, 3, \ldots, m\}$ and achieves 
$$ \frac{\expect{\overline{R}[m]}}{\expect{\overline{T}[m]}} \geq \theta^* - O(\epsilon) $$
must have $m\geq \Omega(1/\epsilon^2)$.  That is, the convergence time is necessarily $\Omega(1/\epsilon^2)$. The work in \cite{neely-renewal-jmlr} developed a Robbins-Monro iterative algorithm with a vanishing stepsize   
to achieve this optimal convergence time. In particular, the algorithm in \cite{neely-renewal-jmlr} achieves
$$\frac{\expect{\overline{R}[m]}}{\expect{\overline{T}[m]}} \geq \theta^* - O(1/\sqrt{m}) \quad \forall m \in \{1, 2, 3, \ldots\} $$
The vanishing stepsize means the algorithm of \cite{neely-renewal-jmlr} cannot adapt to changes. 
The algorithm of the current paper achieves the optimal convergence time using a different technique. The parameter $v$ can be interpreted as an inverse stepsize parameter, so the stepsize is a constant $\epsilon = 1/v$.  With this \emph{constant} stepsize, the algorithm is \emph{adaptive} and achieves reward per unit time within $O(\epsilon)$ of optimality over any consecutive sequence of $\Theta(1/\epsilon^2)$ tasks for which the $\{A[k]\}$ matrices have i.i.d. behavior, regardless of whether the distribution was different before the start of that sequence.

The asymptotic results of Theorem \ref{thm:reward} hold for any $\alpha$ that remains a $\Theta(1)$ constant regardless of the size of $\epsilon$ (the suggestion $\alpha=1$ was made for simplicity).  The value $\alpha$ can be fine tuned. Using $v=1/\epsilon$,  $q=0$ in \eqref{eq:reward} gives
\begin{equation} \label{eq:refined-bound} 
\frac{\expect{\overline{R}[k_0;m]}}{\expect{\overline{T}[k_0;m]}} \geq \theta^* - \epsilon d_1 - \frac{d_2}{\epsilon m} -\frac{(r_{max}/t_{min})}{m} \quad \forall m \in \{1, 2, 3, \ldots\}
\end{equation} 
where 
\begin{align*}
d_1 &= \frac{b + \frac{1}{2\gamma_{min}^2\alpha}(r_{max} + (t_{max} - t_{min})(\beta_1+\beta_2))^2}{t_{min}}\\
d_2&= \frac{\frac{1}{2}(\beta_1+\beta_2)^2 + \frac{\alpha}{2}(\gamma_{max} - \gamma_{min})^2 + \theta^*(\beta_1+\beta_2)}{t_{min}} 
\end{align*}
where 
$$\beta_1+\beta_2 = \frac{1 + r_{max} + \alpha(1/t_{min}-1/t_{max})(t_{max}-t_{min})}{t_{min}} $$
The term $-\epsilon d_1$ in \eqref{eq:refined-bound} does not vanish as $m\rightarrow\infty$. 
Choosing $\alpha>0$ to minimize $d_1$  amounts to minimizing
$$ \frac{1}{\alpha}\left[r_{max} + \frac{(t_{max}-t_{min})(1+r_{max})}{t_{min}} + \alpha\left(\frac{t_{max}-t_{min}}{t_{min}}\right)\left(\frac{t_{max}}{t_{min}}+\frac{t_{min}}{t_{max}}-2\right) \right]^2$$
That is, choose $\alpha>0$ to minimize
$$ \frac{c_1}{\sqrt{\alpha}} + c_2\sqrt{\alpha} $$
where 
\begin{align}
c_1 &= r_{max} + \frac{(t_{max} -t_{min})(1+r_{max})}{t_{min}}  \label{eq:c1} \\
c_2 &=  \left(\frac{t_{max}-t_{min}}{t_{min}}\right)\left(\frac{t_{max}}{t_{min}}+\frac{t_{min}}{t_{max}}-2\right) \label{eq:c2} 
\end{align}
This yields $\alpha = c_1/c_2$. To avoid a very large value of $\alpha$ (which affects the $d_2$ constant) in the special case $c_2<1/2$, one might adjust this to using $\alpha = \frac{c_1}{\max[c_2, 1/2]}$.

\subsection{Sample path limit}

\begin{lem} Suppose the problem \eqref{eq:det1}-\eqref{eq:det3} is feasible with optimal solution $(t^*, r^*, y^*) \in \overline{\Gamma}$ and optimal objective value $\theta^*=r^*/t^*$. Then for any parameters $v>0, \alpha>0$, $q=(q_1, \ldots, q_n)\geq 0$ and 
all positive integers $k_0, m$, our algorithm yields
\begin{align*}
\liminf_{m\rightarrow\infty} \frac{\overline{R}[m]}{\overline{T}[m]} \geq \theta^* -\frac{d_1}{v} \quad \mbox{with prob 1}
\end{align*}
where $d_1$ is given in \eqref{eq:d1}.
\end{lem} 
\begin{proof} 
The proof relates expectations and sample paths 
in a manner that is similar to Theorem 4.4 in \cite{sno-text}. Let $(t^*, r^*,y^*) \in \overline{\Gamma}$ solve \eqref{eq:det1}-\eqref{eq:det3} and let $\theta^*=r^*/t^*$. For $k \in \{1, 2, 3, \ldots\}$ define 
\begin{align*}
X[k] &= \Delta[k] - vR[k] - \frac{\alpha v^2}{2}(1/t^*-\gamma[k-1])^2 + \frac{\alpha v^2}{2}(1/t^*-\gamma[k])^2\\
W[k] &= \expect{X[k]|H[k]} -X[k]
\end{align*}
where $H[k]=(A[1], \ldots, A[k-1])$, and we observe that $X[1], \ldots, X[k-1]$ are determined from the information $H[k]$. Next observe that there is a positive number $z$ such 
that for each $k \in \{1, 2, 3, \ldots\}$ we have 
\begin{itemize} 
\item $\expect{W[k]|H[k]} = 0$. 
\item $|W[k]|\leq z$.
\end{itemize} 
where the final fact holds because all virtual queues are deterministically bounded. 
For each positive integer $m$ define $\overline{W}[m] = \frac{1}{m}\sum_{k=1}^mW[k]$. It follows by the law of large numbers for martingale 
differences  (see \cite{chow-lln}) that  
\begin{equation} \label{eq:martingale} 
\lim_{m\rightarrow\infty} \overline{W}[m]=0 \quad \mbox{ with prob 1}
\end{equation}
We have 
\begin{equation} \label{eq:zbar}
\overline{W}[m] = \left(\frac{1}{m}\sum_{k=1}^m\expect{X[k]|H[k]}\right)-\overline{X}[m] 
\end{equation} 
By definition of $X[k]$ we obtain 
\begin{align}
\overline{X}[m]&=\frac{1}{m}\sum_{k=1}^m X[k]\nonumber \\
&= \frac{L[m+1]-L[1]}{m} -v\overline{R}[m] -\frac{\alpha v^2}{2m}(1/t^*-\gamma[0])^2 + \frac{\alpha v^2}{2m}(1/t^*-\gamma[m])^2 \nonumber\\
&=-v\overline{R}[m] + G_1[m] \label{eq:Xk}
\end{align}
where $G_1[m]$ is a random sequence that satisfies
\begin{equation} \label{eq:G1sure}
 \lim_{m\rightarrow\infty} G_1[m]=0 \quad \mbox{(surely)}
 \end{equation} 
which holds because all virtual queues are deterministically bounded, as are values $L[m+1]$ and $(1/t^*-\gamma[m])^2$.

We have by \eqref{eq:DPP} 
that for each positive integer $k\geq 2$: 
\begin{align*}
\expect{X[k]|H[k]} &\leq b - \frac{v\theta^*}{\gamma[k-1]} + \frac{(vr_{max} + cv\norm{q} + v(t_{max}-t_{min})(\beta_1+\beta_2))^2}{2\gamma_{min}^2\alpha v^2} \\
&=d_1t_{min} - v\theta^*\frac{1}{\gamma[k-1]} \\
&= d_1t_{min} - v\theta^*T[k-1] + v\theta^*(T[k-1]-1/\gamma[k-1])\\
&\leq d_1t_{min} - v\theta^*T[k-1] + v\theta^*(J[k]-J[k-1]) 
\end{align*}
where the first inequality uses the definition of $d_1$ in \eqref{eq:d1}; the final inequality uses \eqref{eq:jt-bound}; we have assumed $k\geq 2$ so that $T[k-1]$ makes sense on the right-hand-side.  Summing over $k \in \{2, \ldots, m\}$ and using $d_1t_{min}(m-1)/m\leq d_1t_{min}$ gives 
\begin{align}
\frac{1}{m}\sum_{k=1}^m\expect{X[k]|H[k]} &\leq \frac{\expect{X[1]}}{m}+d_1t_{min} - v\theta^*\overline{T}[m] + \frac{v\theta^*T[m]}{m}+v\theta^*\frac{(J[m]-J[1])}{m}\nonumber\\
&= -v\theta^*\overline{T}[m] + d_1t_{min} + G_2[m] \label{eq:G2}
\end{align} 
where $G_2[m]$ is a random process that satisfies 
\begin{equation} \label{eq:G2sure}
 \lim_{m\rightarrow\infty} G_2[m] = 0 \quad \mbox{(surely)} 
 \end{equation} 
Subtracting $\overline{X}[m]$  from both sides of  \eqref{eq:G2} and using \eqref{eq:zbar} gives 
\begin{align*}
\overline{W}[m] &\leq - v\theta^*\overline{T}[m] + d_1t_{min} + G_2[m] - \overline{X}[m]\\
&= -v\theta^*\overline{T}[m] + d_1t_{min} + G_2[m] + v\overline{R}[m] - G_1[m]
\end{align*}
where the final equality uses \eqref{eq:Xk}. 
Dividing both sides by $v\overline{T}[m]$ and rearranging terms gives 
\begin{align*}
\frac{\overline{R}[m]}{\overline{T}[m]} - \theta^* &\geq  - \frac{d_1t_{min}}{v\overline{T}[m]} + \frac{\overline{W}[m]}{v\overline{T}[m]} + \frac{G_1[m]-G_2[m]}{v\overline{T}[m]}\\
&\geq -\frac{d_1}{v} + \frac{-|\overline{W}[m]|}{vt_{min}} - \frac{|G_1[m]-G_2[m]|}{vt_{min}} 
\end{align*}
Taking $m\rightarrow\infty$ and using \eqref{eq:martingale}, \eqref{eq:G1sure}, \eqref{eq:G2sure} gives (with prob 1)
$$ \liminf_{m\rightarrow\infty} (\overline{R}[m]/\overline{T}[m] - \theta^*) \geq - \frac{d_1}{v} $$
\end{proof}

\section{Constraints} \label{section:constraints} 

This section considers the process $Y[k] \in \mathbb{R}^n$, where $n$ is a given positive integer (the case $n=0$ is considered in Subsection \ref{section:n=0}). The hierarchical nature of our algorithmic decision for each task $k$ allows an analysis of the virtual queues $Q[k]$ separately from the $\gamma[k]$ decisions.   Define
$$Z[k] = \norm{Q[k]} \quad \forall k \in \{1, 2, 3, \ldots\}$$
Recall that $H[k] = (A[1], \ldots, A[k-1])$ and knowledge of $H[k]$ determines $Q[k]$ and $Z[k]$ (that is, $Q[k]$ and $Z[k]$ are $H[k]$-measurable). 

\begin{thm} \label{thm:constraint1} Assume the Slater condition \eqref{eq:Slater} holds for some $s>0$ and vector $(t^s, r^s, y^s) \in \overline{\Gamma}$.  
Then for all $k \in \{1, 2, 3, \ldots\}$ our algorithm gives 
\begin{equation} \label{eq:norm-bound}
\expect{Z[k+1]-Z[k]|H[k]} \leq \left\{\begin{array}{cc}
c & \mbox{ if $Z[k]<\lambda$} \\
-s/2 & \mbox{ if $Z[k]\geq \lambda$} 
\end{array}\right.
\end{equation}
and 
\begin{equation}\label{eq:Zc}
 |Z[k+1]-Z[k]|\leq c  
 \end{equation} 
where $c$ is defined in \eqref{eq:bound3} and $\lambda$ is defined 
\begin{align} 
\lambda&= \max\left\{\frac{vd_0}{s} - \frac{s}{4}, \frac{s}{2}\right\} \label{eq:lambda} \\
d_0&= 2r_{max} + 2(\beta_1+\beta_2)(t_{max}-t_{min}) + c^2/v \label{eq:d0} 
\end{align} 
\end{thm} 
\begin{proof} 
Fix $k \in \{1, 2, 3, \ldots\}$.  To prove \eqref{eq:Zc}, we have 
\begin{align}
Z[k+1] &=  \norm{Q[k+1]} \nonumber\\
&\overset{(a)}{\leq} \norm{Q[k]+Y[k]} \nonumber\\
&\leq \norm{Q[k]} + \norm{Y[k]} \nonumber\\
&\leq Z[k] + c \label{eq:mevan1} 
\end{align}
where inequality (a) holds by the queue update \eqref{eq:q-update} and the nonexpansion property of projections; the 
final inequality uses $\norm{Y[k]}\leq c$ from \eqref{eq:bound3}. 
Similarly, 
\begin{align}
\norm{Q[k+1]-Q[k]}^2 &= \sum_{i=1}^n (Q_i[k+1]-Q_i[k])^2\nonumber\\
&\overset{(a)}{=}\sum_{i=1}^n([Q_i[k]+Y_i[k]]_0^{q_iv} - [Q_i[k]]_0^{q_iv})^2\nonumber\\
&\overset{(b)}{\leq}\sum_{i=1}^n(Q_i[k]+Y_i[k]-Q_i[k])^2\nonumber\\
&=\norm{Y[k]}^2 \nonumber\\
&\overset{(c)}{\leq} c^2 \label{eq:mevan3} 
 \end{align}
where (a) holds by substituting the definition of $Q_i[k+1]$ from \eqref{eq:q-update} and the fact $Q_i[k]=[Q_i[k]]_0^{v_iq}$; (b) holds by the nonexpansion property of projections; (c) holds because $\norm{Y[k]}\leq c$ from  \eqref{eq:bound3}.
Furthermore 
\begin{align}
Z[k+1] &= \norm{Q[k+1]} \nonumber\\
&\geq \norm{Q[k]} - \norm{Q[k+1]-Q[k]} \nonumber\\
&\geq Z[k] - c \label{eq:mevan2}
\end{align}
where the final inequality holds by \eqref{eq:mevan3}. 
The inequalities \eqref{eq:mevan1} and \eqref{eq:mevan2} together prove  \eqref{eq:Zc}.

We now prove \eqref{eq:norm-bound}. The case $Z[k]<\lambda$ follows immediately from \eqref{eq:Zc}. It suffices to consider $Z[k]\geq \lambda$. The queue update \eqref{eq:q-update} ensures
\begin{equation} \label{eq:q-implies}
\norm{Q[k+1]}^2  \leq  \norm{Q[k]}^2 + c^2 + 2Q[k]^{\top}Y[k]
\end{equation}
The Slater condition holds and so there is a decision vector  $(T^*[k], R^*[k], Y^*[k]) \in Row(A[k])$ that satisfies (with prob 1):
\begin{equation} \label{slater-consequence}
\expect{(T^*[k], R^*[k], Y^*[k])|H[k]} = (t^s, r^s, y^s) 
\end{equation} 
By \eqref{eq:analysis1} we have 
\begin{align*}
-vR[k] + J[k]T[k] + Q[k]^{\top}Y[k] \leq -vR^*[k] + J[k]T^*[k] + Q[k]^{\top}Y^*[k]
\end{align*}
Multiplying the above inequality by $2$ and  rearranging terms gives 
\begin{align*}
2Q[k]^{\top}Y[k] &\leq 2v(R[k]-R^*[k]) + 2J[k](T^*[k]-T[k])+ 2Q[k]^{\top}Y^*[k]\\
&\leq 2vr_{max} + 2v(\beta_1+\beta_2)(t_{max}-t_{min}) + 2Q[k]^{\top}Y^*[k]
\end{align*}
where the final inequality uses $J[k]\leq v(\beta_1+\beta_2)$.  Substituting this into the right-hand-side
of \eqref{eq:q-implies} gives 
\begin{align*}
\norm{Q[k+1]}^2 &\leq  \norm{Q[k]}^2 +c^2 + 2vr_{max} + 2v(\beta_1+\beta_2)(t_{max}-t_{min}) + 2Q[k]^{\top}Y^*[k]\\
&= \norm{Q[k]}^2 + vd_0 + 2Q[k]^{\top}Y^*[k]\\
&= Z[k]^2 + vd_0 + 2Q[k]^{\top}Y^*[k]
\end{align*}
where $d_0$ is defined in \eqref{eq:d0}. 
Taking conditional expectations of both sides and using \eqref{slater-consequence} gives (with prob 1)
\begin{align*}
\expect{\norm{Q[k+1]}^2 |H[k]} &\leq Z[k]^2 +  vd_0 + 2Q[k]^{\top}y^s\\
&\overset{(a)}{\leq} Z[k]^2+ vd_0 - 2s \sum_{i=1}^n Q_i[k] \\
&\overset{(b)}{\leq} Z[k]^2+ vd_0 - 2s Z[k] 
\end{align*}
where inequality (a) holds by \eqref{eq:Slater}; inequality (b) holds by the triangle inequality 
$$Z[k]=\norm{Q[k]} \leq \sum_{i=1}^nQ_i[k]$$
Jensen's inequality and the definition $Z[k+1]=\norm{Q[k+1]}$ gives 
$$ \expect{Z[k+1]|H[k]}^2 \leq \expect{\norm{Q[k+1]}^2|H[k]} $$ 
Substituting this into the previous inequality gives 
\begin{align*}
\expect{Z[k+1]|H[k]}^2 &\leq  Z[k]^2 + vd_0 - 2sZ[k]\\
&= (Z[k]-s/2)^2 - s^2/4 + vd_0 - sZ[k]\\
&\overset{(a)}{\leq} (Z[k]-s/2)^2 - s^2/4 + vd_0-s\lambda \\
&\overset{(b)}{\leq} (Z[k]-s/2)^2
\end{align*}
where (a) holds because we assume $Z[k]\geq \lambda$; (b) holds because $\lambda \geq vd_0/s - s/4$ by definition of $\lambda$ in \eqref{eq:lambda}. Definition of $\lambda$ also implies $\lambda \geq s/2$. Since $Z[k]\geq \lambda \geq s/2$ we can take square roots to obtain 
$$ \expect{Z[k+1]|H[k]} \leq Z[k]-s/2$$
\end{proof} 

The above theorem is in the form required of Lemma 4 in \cite{neely-energy-convergence-ton} and so we obtain the following corollary:  

\begin{cor} \label{cor:1} Assume the Slater condition \eqref{eq:Slater} holds for some $s>0$ and vector $(t^s, r^s, y^s) \in \overline{\Gamma}$.  
Then for all $k_0 \in \{1, 2, 3, \ldots\}$ and all $z_0 \in \times_{i=1}^n [0, vq_i]$, given $Z[k_0]=z_0$, 
our algorithm gives (with probability 1):
\begin{equation} \label{eq:cor}
\expect{e^{\eta Z[k]}|H[k_0]} \leq d + (e^{\eta z_0}-d)\rho^{k-k_0} \quad \forall k \in \{k_0, k_0+ 1, k_0+2,  \ldots\}
\end{equation} 
where 
\begin{align}
\eta &=\frac{s/2}{c^2 + cs/6}\label{eq:eta}\\
\rho &= 1 - \eta s/4 \label{eq:rho}\\
d &= \frac{(e^{\eta c}  -\rho)e^{\eta \lambda}}{1-\rho} \label{eq:d}
\end{align}
where $\lambda$ is given in \eqref{eq:lambda}. Further, it holds that $s/2\leq c$, $e^{\eta c} \leq \rho$, and 
$0<\rho<1$.  
\end{cor} 

\begin{proof} 
This follows by applying Lemma 4 in \cite{neely-energy-convergence-ton} to the result of Theorem \ref{thm:constraint1}.
\end{proof}

\

The above corollary implies that the distribution of $Z[k]=\norm{Q[k]}$ decays exponentially fast.
This is useful to ensure the events  $Q_i[k]> vq_i - y_{i,max}$ are rare (meaning $1_i[k]=1$ is rare, which is 
important for satisfying the constraints because of Lemma \ref{lem:vq-bounds}). 
It suffices to choose $q=(q_1, \ldots, q_n)$ as a constant that does not scale with the parameter $v$, but that is 
suitably large.  For simplicity we choose $q_i$ to be the same for all $i\in\{1, \ldots, n\}$. 

\

\begin{thm} \label{thm:constraint2} Assume the Slater condition \eqref{eq:Slater} holds for some $s>0$ and vector $(t^s, r^s, y^s) \in \overline{\Gamma}$. Fix $q_1\geq 2d_0/s$ and fix $q_i=q_1$ for $i \in \{1, \ldots, n\}$.  
 Fix $\epsilon>0$.  
Assume $v= \max\{1/\epsilon, \frac{3s^2}{4d_0}\}$. 
For all $i \in \{1, \ldots, n\}$, all $k_0\in \{1, 2, 3, \ldots\}$, and all $m\geq \lceil\frac{4vq_1\sqrt{n}}{s}\rceil + v^2$ we have (with probability 1):
$$ \frac{1}{m}\sum_{k=k_0}^{k_0+m-1} \expect{Y_i[k]|H[k_0]} \leq O(\epsilon)$$
\end{thm} 

\begin{proof} 
Since queues are bounded, we have $Q_i[k_0] \in [0, vq_i]$ for all $i \in\{1, \ldots, n\}$, and so 
$$ \mbox{$\norm{Q[k_0]}\leq v\sqrt{\sum_{i=1}^nq_i^2} = vq_1\sqrt{n}$}$$
Fix $z_0=\norm{Q[k_0]}$ and note that $z_0\in [0, vq_1\sqrt{n}]$. 
Define $k_1 = \lceil\frac{4vq_1\sqrt{n}}{s}\rceil$. 
Fix $i \in \{1, \ldots, n\}$, $k_0 \in \{1, 2, 3, \ldots\}$, $m \geq k_1 + v^2$. 
From \eqref{eq:vq1} and the fact $q_i=q_1$ we have 
\begin{align*}
\frac{1}{m}\sum_{k=k_0}^{k_0+m-1}Y_i[k] &\leq \frac{q_1v}{m} + \frac{y_{i,max}}{m}\sum_{k=k_0}^{k_0+k_1-1} 1_i[k] +  \frac{y_{i,max}}{m}\sum_{k=k_0+k_1}^{k_0+m-1} 1_i[k]\\
&\leq \frac{q_1v}{m} + \frac{y_{i,max}k_1}{m} + \frac{y_{i,max}}{m}\sum_{k=k_0+k_1}^{k_0+m-1} 1_i[k]
\end{align*}
Since $q_1$ is a $\Theta(1)$ constant that does not scale with $\epsilon$, and since $m\geq k_1+v^2$, we have 
\begin{align*}
&\frac{q_1v}{m} \leq O(\epsilon) \\
 &\frac{y_{i,max}k_1}{m} \leq O(\epsilon) 
\end{align*}
It suffices to show 
\begin{equation} \label{eq:suff1}
\frac{1}{m}\sum_{k=k_0+k_1}^{k_0+m-1}\expect{1_i[k]|H[k_0]} \leq O(\epsilon)
\end{equation} 
To this end, observe by definition of $1_i[k]$ in \eqref{eq:one-i} 
$$e^{\eta (q_iv - y_{i,max})} 1_i[k] \leq e^{\eta Q_i[k]} \leq e^{\eta Z[k]} $$
Taking conditional expectations and using \eqref{eq:cor} gives for all $k\geq k_0+k_1$: 
\begin{align*}
e^{\eta (q_iv - y_{i,max})}\expect{1_i[k]|H[k_0]} &\leq d + (e^{\eta z_0} - d)\rho^{k-k_0} \\
&\leq d + e^{\eta z_0} \rho^{k-k_0}
\end{align*}
Summing over the (fewer than $m$) terms and dividing by $m$ gives 
\begin{align*}
e^{\eta (q_iv - y_{i,max})}\frac{1}{m}\sum_{k=k_0+k_1}^{k_0+m-1} \expect{1_i[k]|H[k_0]} &\leq d + \frac{e^{\eta z_0}\rho^{k_1}}{m}\sum_{k=k_0+k_1}^{k_0+m-1} \rho^{k-k_1-k_0}\\
&\leq d + \frac{e^{\eta z_0}\rho^{k_1}}{m}\frac{1}{1-\rho}
\end{align*}
Recall that $q_i=q_1$ for all $i$. To show \eqref{eq:suff1} it suffices to show 
\begin{align}
&d e^{-\eta (q_1v - y_{i,max})} \leq O(\epsilon) \label{eq:suff2}\\
&e^{-\eta (q_1v - y_{i,max})}  \frac{e^{\eta z_0}\rho^{k_1}}{(1-\rho)m} \leq O(\epsilon) \label{eq:suff3}
\end{align}
In fact we show both these terms are \emph{much smaller} than $O(\epsilon)$. 

By assumption, $v\geq \frac{3s^2}{4d_0}$ and so  from \eqref{eq:lambda}
\begin{equation} \label{eq:lambda-c}
\lambda = \frac{vd_0}{s} - \frac{s}{4}
\end{equation} 
By definition of $d$ in \eqref{eq:d}: 
\begin{align*}
de^{-\eta (q_1v - y_{i,max})} &= \frac{(e^{\eta c} - \rho)e^{\eta \lambda}}{1-\rho} e^{\eta y_{i,max}}e^{-\eta q_1v}\\
&\overset{(a)}{\leq} \frac{1}{1-\rho} e^{\eta(c + \lambda + y_{i,max} - q_1v)} \\
&\overset{(b)}{=} \frac{1}{1-\rho} e^{\eta(c + vd_0/s - s/4 + y_{i,max} - q_1v)} \\
&=\left(\frac{e^{\eta (y_{i,max} + c-s/4)}}{1-\rho}\right)e^{-\eta v(q_1-d_0/s)}\\
&\overset{(c)}{\leq}\left(\frac{e^{\eta (y_{i,max} + c-s/4)}}{1-\rho}\right)e^{-\eta vd_0/s}\\
&\overset{(d)}{\leq} \left(\frac{e^{\eta (y_{i,max} + c-s/4)}}{1-\rho}\right)e^{-(\eta d_0/s)/\epsilon}\\
&\overset{(e)}{\leq} O(e^{-(\eta d_0/s)/\epsilon})\\
&\leq O(\epsilon) 
\end{align*}
where (a) holds because $0<\rho<1$ (recall Corollary \ref{cor:1}); 
(b) holds by \eqref{eq:lambda-c}; 
(c) holds because $q_1\geq 2d_0/s$; (d) holds because $v\geq 1/\epsilon$; (e) holds 
because  $y_{i,max},\eta, c, s, \rho$ are all $\Theta(1)$ constants that do not scale with $\epsilon$. The term goes to zero exponentially fast as $\epsilon\rightarrow 0$, much faster than $O(\epsilon)$. This proves 
  \eqref{eq:suff2}.
  
 To show \eqref{eq:suff3}, we have 
 \begin{align}
 e^{-\eta (q_1v - y_{i,max})}\frac{e^{\eta z_0}\rho^{k_1}}{(1-\rho)m}&\leq \frac{e^{\eta y_{i,max}}}{(1-\rho)m} e^{\eta z_0}\rho^{k_1}\nonumber\\
 &=\frac{e^{\eta y_{i,max}}}{(1-\rho)m} e^{\eta z_0 + k_1\log(\rho)}\label{eq:obs4}
 \end{align}
By definition of $\rho$ in \eqref{eq:rho} we have 
\begin{align*}
k_1 \log(\rho) &= k_1 \log(1-\eta s/4) \\
&\leq-k_1 \eta s/4
\end{align*}
which uses the fact $\log(1+x)\leq x$ for all $x>-1$ (recall from Corollary \ref{cor:1} that $\eta s/4<1$). 
Adding $\eta z_0$ to both sides gives 
\begin{align*}
\eta z_0 + k_1\log(\rho) &\leq \eta(z_0 - k_1s/4) \\
&\overset{(a)}{\leq} \eta (vq\sqrt{n} - k_1s/4)\\
&\overset{(b)}{\leq} 0
\end{align*}
 where (a) holds because $z_0^2 \leq n(vq)^2$; (b) holds by definition of $k_1$. Substituting the above inequality into \eqref{eq:obs4}
 gives 
 \begin{align*}
 e^{-\eta (q_1v - y_{i,max})}\frac{e^{\eta z_0}\rho^{k_1}}{(1-\rho)m}&\leq \frac{e^{\eta y_{i,max}}}{(1-\rho)m} \\
 &\leq \frac{e^{\eta y_{i,max}}}{(1-\rho)(k_1+v^2)} \\
 &\leq O(\epsilon^2)\\
 &\leq O(\epsilon)
  \end{align*}
  where we have used $v\geq 1/\epsilon$.
\end{proof} 

\subsection{Discussion} 

The case $n=0$ has no penalties $Y_i[k]$ and the algorithm uses a single parameter $v>0$ that is scaled (using $v=1/\epsilon$) for a tradeoff between adaptation time and proximity to the optimal solution. When $n>0$ the algorithm has a parameter $v>0$ and another parameter $q_1>0$.  Theorem \ref{thm:constraint2} suggests $q_1\geq 2d_0/s$. This requires rough knowledge of $d_0/s$, where $s$ is the Slater parameter.  In practice there is little danger in choosing $q_1$ to be too large. Indeed, even choosing $q_1=\infty$ works well in practice. Intuitively, this is because the virtual queue update \eqref{eq:q-update} for $q_1=\infty$  reduces to 
$$Q_i[k+1] = \max\left\{Q_i[k] + Y_i[k], 0\right\} \quad \forall k \in \{1, 2, 3, \ldots\}$$
which means $1_i[k]=0$ for all $k$ and the inequality \eqref{eq:vq1} can be modified to 
\begin{equation}\label{eq:discuss} 
\frac{1}{m}\sum_{k=k_0}^{k_0+m-1} Y_i[k] \leq \frac{Q_i[k_0+m]-Q_i[k_0]}{m}
\end{equation} 
for all positive integers $k_0, m$.  Intuitively, the Slater condition still ensures a negative drift condition similar to \eqref{eq:norm-bound}, so that $\norm{Q[k]}$ is still concentrated 
so it is rarely much larger than the $\lambda$ parameter in \eqref{eq:vq1}, where $\lambda = \Theta(v)$. Intuitively, while the $J[k]$ virtual queue would no longer be 
deterministically bounded, it would stay within its existing bounds with high probability. Taking expectations of \eqref{eq:discuss} would then produce a right-hand-side proportional to $v/m$, which is $O(\epsilon)$ whenever $v=1/\epsilon$ and $m\geq 1/\epsilon^2$.
We do not pursue this line of analysis because our use of finite $q_i$ values enables strong deterministic bounds on $Q_i[k]$ and $J[k]$. Thus, our analysis over any sequence of tasks $\{k_0, \ldots, k_0+m-1\}$ indeed holds regardless of the history of the system before task $k_0$, including the (rare) cases when the virtual queues hit their upper bound values before task $k_0$.

\begin{figure}[t]
   \centering
   \includegraphics[width=4in]{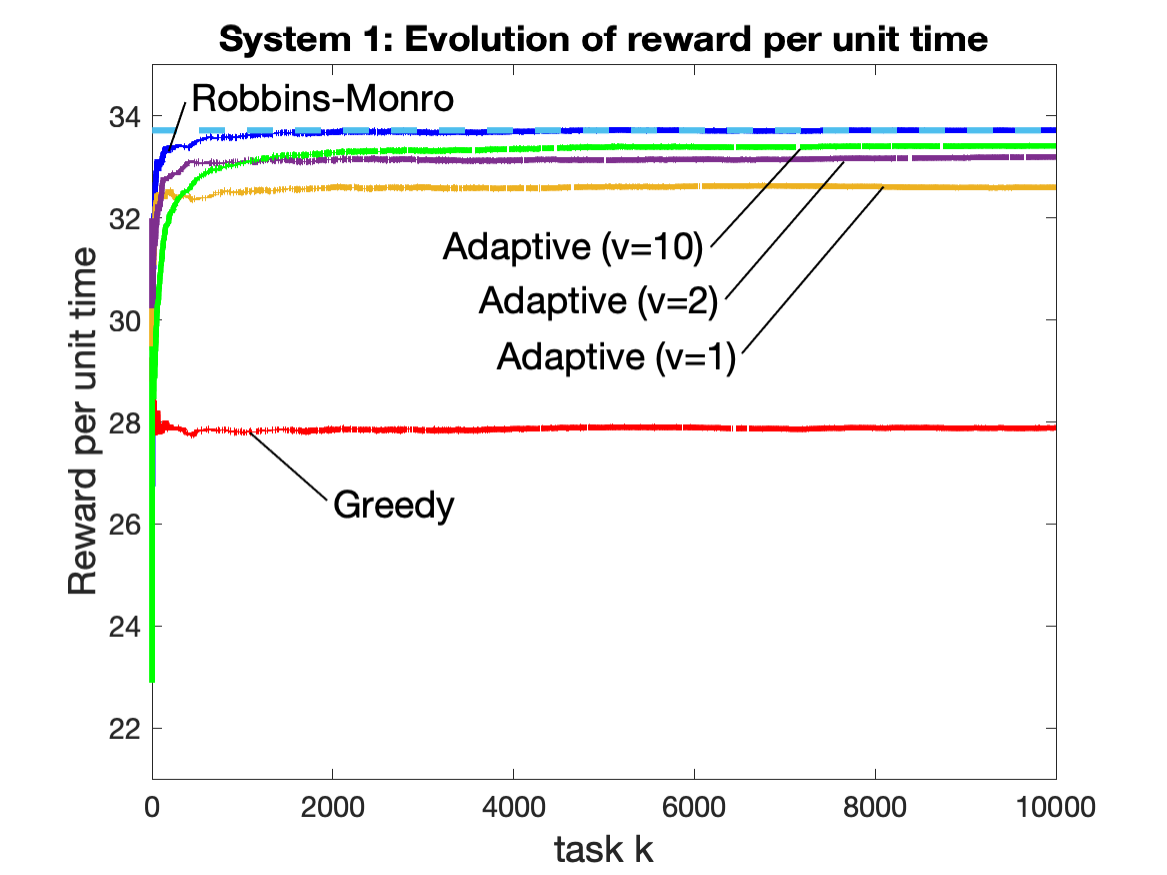} 
   \caption{System 1: Accumulated reward per unit time for the proposed adaptive algorithm (with $v\in\{1,2,10\}$), the vanishing-stepsize Robbins-Monro algorithm; and the greedy algorithm. All data points are averaged over $40$ independent simulations.}
   \label{fig:1}
\end{figure}

\begin{figure}[t]
   \centering
   \includegraphics[width=4in]{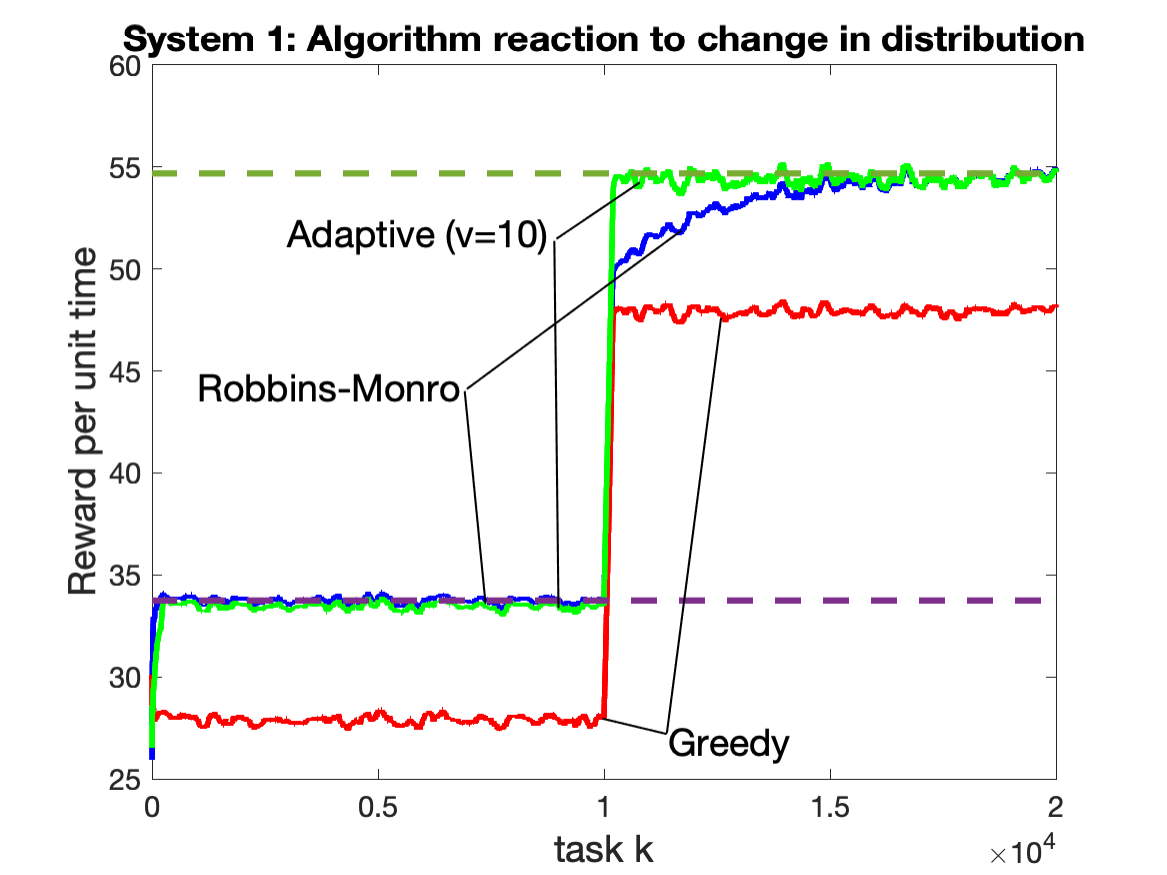} 
   \caption{System 1: Testing adaptation over a simulation of $2\times 10^4$ tasks with a distributional change introduced at the halfway point (task $10^4$). The two horizontal dashed lines represent optimal $\theta^*$ values for the two distributions. Each point for task $k_0$ is the result of a moving window average $\frac{\sum_{k=1}^{200}\expect{R[k_0-k]}}{\sum_{k=1}^{200}\expect{T[k_0-k]}}$, where expectations are obtained by averaging over $40$ independent simulations. The adaptive algorithm (with $v=10$) quickly adapts to the change. The Robbins-Monro algorithm takes a long time to adapt.}
   \label{fig:2}
\end{figure}

\section{Simulation}

\subsection{System 1} 

This subsection considers sequential project selection with the  
goal of maximizing reward per unit time (with no 
penalty processes $Y_i[k]$). The system is similar to one simulated in  \cite{neely-renewal-jmlr}.  The i.i.d. matrices $\{A[k]\}$ have two columns and a random number of rows. The number of rows is equal to number of project options for task $k$.  Two different distributions for $A[k]$ are considered in the simulations (specified at the end of this subsection). Both distributions have  $t_{min}=1, t_{max}=10, r_{max}=500$.

Fig \ref{fig:1} illustrates results for a simulation over $10^4$ tasks using i.i.d. $\{A[k]\}$ with Distribution 1. The vertical axis in Fig. \ref{fig:1} represents 
the accumulated reward per task starting with task $1$ and running up to the current task $k$: 
$$ \frac{\sum_{j=1}^k\expect{R[j]}}{\sum_{j=1}^k\expect{T[j]}}$$
where the expectations $\expect{R[j]}$ and $\expect{T[k]}$ are approximated by averaging over $40$ independent simulation 
runs. 
 Fig. \ref{fig:1} compares the \emph{greedy} algorithm of always choosing the task $k$ that maximizes the instantaneous $R[k]/T[k]$ value; the (nonadaptive) Robbins-Monro algorithm from \cite{neely-renewal-jmlr} that uses a stepsize $\eta[k]=\frac{1}{k+1}$; the proposed adaptive algorithm for the cases $v=1, v=2, v=10$ (and using $\alpha = c_1/\max[c_2, 1/2]$).  
The dashed horizontal line in Fig. \ref{fig:1} is the optimal $\theta^*$ value corresponding to Distribution 1. The value $\theta^*$ is difficult to calculate analytically, so we use an empirical value  obtained by the final point on the Robbins-Monro curve. It can be seen that the greedy algorithm has significantly worse performance compared to the others. The Robbins-Monro algorithm, which uses a vanishing stepsize, has the fastest convergence and the highest achieved reward per unit time.  As predicted by our theorems, the proposed adaptive 
algorithm has convergence time that gets slower as $v$ is increased, with a corresponding tradeoff in \emph{accuracy}, where accuracy relates to the proximity of the converged value to the optimal $\theta^*$.  The case $v=1$ converges quickly but has less accuracy. The cases $v=2$ and $v=10$ have accuracy that is competitive 
with Robbins-Monro.  

Fig. \ref{fig:2} illustrates the adaptation advantages of the proposed algorithm. Figure \ref{fig:2} considers simulations over $2\times10^4$ tasks.  The first half of the simulation refers to tasks $\{1, \ldots, 10^4\}$, the second half refers to tasks $\{10^4+1, \ldots, 2\times 10^4\}$. The $\{A[k]\}$ matrices in the first half are i.i.d. with Distribution 1; in the second half they are  i.i.d. with Distribution 2.  Nobody tells the algorithms that a change occurs at the halfway mark, rather, the algorithms must adapt.  
The two dashed horizontal lines represent optimal $\theta^*$ values for Distribution 1 and Distribution 2.  Data in Fig. \ref{fig:2} is plotted as a moving average with a window of the past 200 tasks (and averaged over 40 independent simulations). 
As seen in the figure, the adaptive algorithm (with $v=10$) produces near optimal performance that quickly adapts to the change. In stark contrast, 
the Robbins-Monro algorithm adapts very slowly to the change and takes roughly $(3/4)\times 10^4$ tasks to move close to optimality. The adaptation time of Robbins-Monro is much slower than its convergence time starting at task $1$. This is due to the vanishing stepsize  and the fact that, at the time of the distribution change, the stepsize is very small. 
Theoretically, the Robbins-Monro algorithm has an \emph{arbitrarily large} adaptation time, as can be seen by imagining
a simulation that uses a fixed distribution for a number of tasks $x$ before changing to another distribution: The stepsize at the time of change is $\eta[x]=1/(x+1)$, hence an arbitrarily large value of $x$ yields an arbitrarily large adaptation time. 

Fig. \ref{fig:2} shows the greedy algorithm adapts very quickly. This is because the greedy algorithm maximizes $R[k]/T[k]$ for each task $k$ without regard to history. Of course, the greedy algorithm is the least accurate and produces results that are significantly less than optimal for both distributions. To avoid clutter, the adaptive algorithm for 
cases $v=1, v=2$ are not plotted in Fig. \ref{fig:2}.  Only the case $v=10$ is shown because this case has the slowest adaptation but the most accuracy  (as compared to $v=1, v=2$ cases).  While not shown in Fig. \ref{fig:2}, it was observed that the accuracy of the $v=2$ case was only marginally worse than that of the $v=10$ case (similar to Fig. \ref{fig:1}). 

The two distributions used in Fig. \ref{fig:1} and Fig. \ref{fig:2} are:

\begin{itemize} 
\item Distribution 1:  With $M[k]$ being the random number of rows, we use $P[M[k]=1]=0.1$, $P[M[k]=2]=0.6$, $P[M[k]=3]=0.15$, $P[M[k]=4]=0.15$. The first row is always $[T_1, R_1]=[1, 0]$ and represents a ``vacation'' option that lasts for one unit of time and has zero reward (as explained in \cite{neely-renewal-jmlr}, it can be optimal to take vacations a certain fraction of time, even if there are other row options). The remaining rows $r$, if any, have parameters $[T_r, R_r]$ generated independently with $T_r\sim Unif[1,10]$ and $R_r=T_rG_r$ where $G_r \sim Unif[0,50]$ and is independent of $T_r$. 

\item Distribution 2:  We use $P[M[k]=1]=0$, $P[M[k]=2]=0.2$, $P[M[k]=3]= 0.4$, $P[M[k]=4]=0.4$. 
The first row is always $[T_1,R_1]=[1, 0]$. The other rows $r$ are independently chosen as a random vector $[T_r,R_r]$ with $T_r\sim Unif[1,10]$, $R_r=G_rT_r+H_r$ with $G_r, H_r$  independent
and $G_r\sim Unif[10,30]$, $H_r\sim Unif[0,200]$. 
\end{itemize}

\subsection{System 2} 

 \begin{figure}[htbp]
   \centering
   \includegraphics[width=4in]{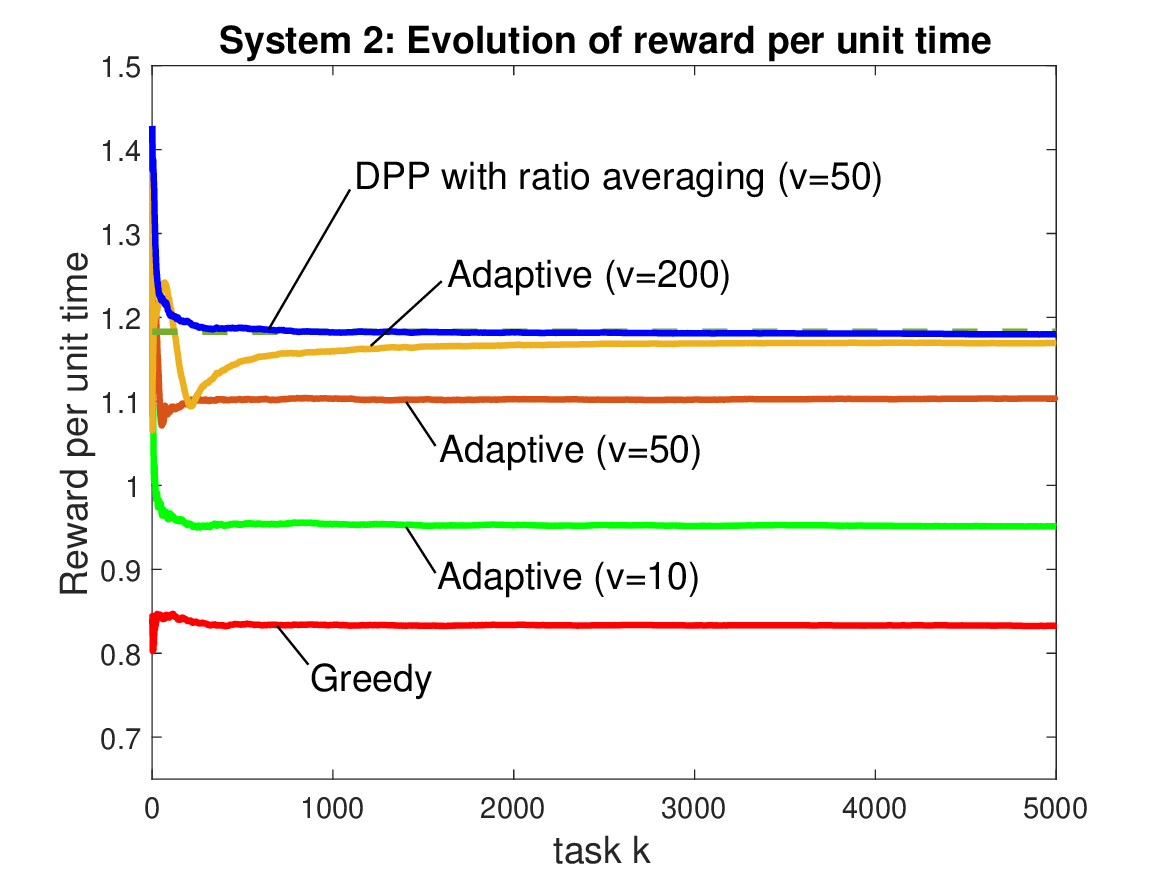} 
   \caption{Time average reward up to task $k$ for the adaptive algorithm ($v\in\{10,50,200\}$); the DPP algorithm with ratio averaging; the greedy algorithm.}
   \label{fig:Rewardnew}
\end{figure}

\begin{figure}[htbp]
   \centering
   \includegraphics[width=4in]{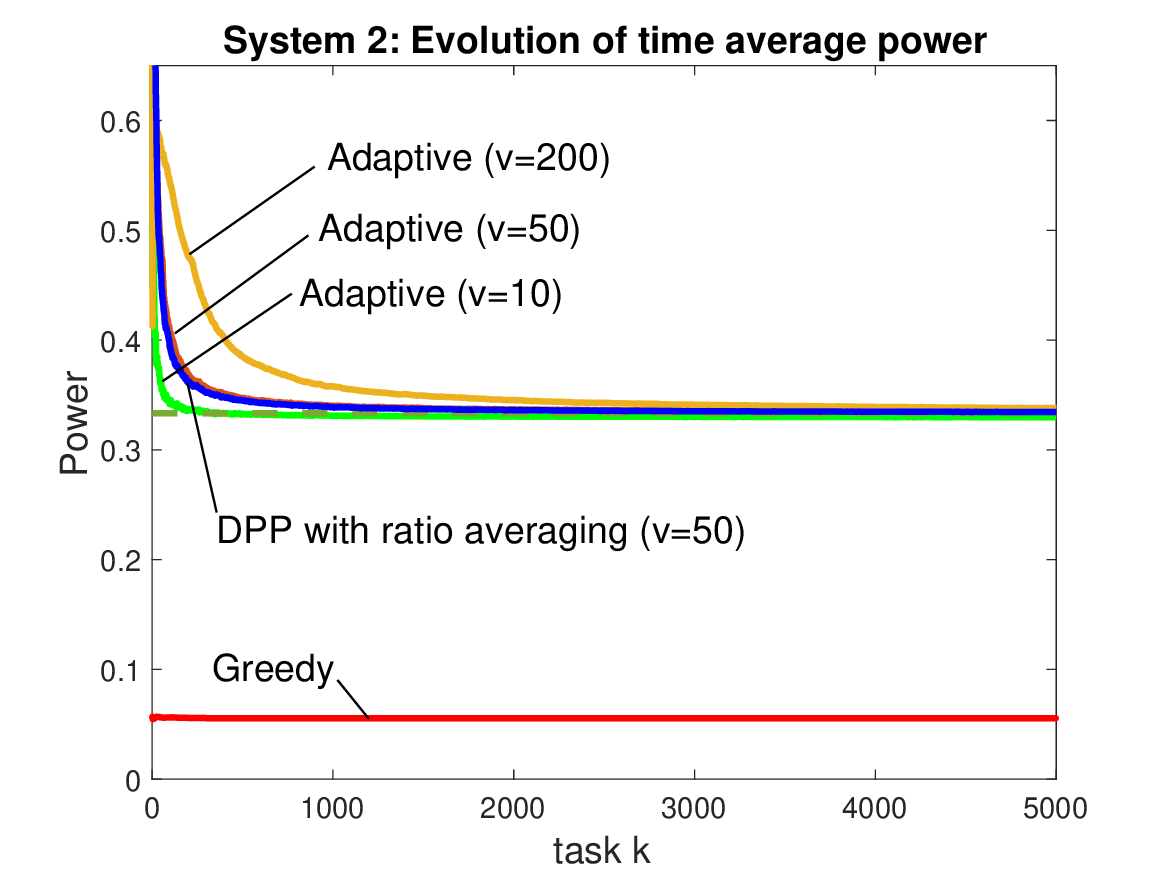} 
   \caption{Corresponding time averaged power for the simulations of Fig. \ref{fig:Rewardnew}. The horizontal asymptote is $p_{av}=1/3$.}
   \label{fig:Powernew}
\end{figure}

This subsection considers a device that processes computational tasks with the goal of maximizing time average profit subject to a time average
power constraint of $p_{av}=1/3$ energy/time. There is a penalty process $Y[k]$ and so  
the Robbins-Monro algorithm of \cite{neely-renewal-jmlr} cannot be used.  We compare the adaptive algorithm of the current
paper the drift-plus-penalty ratio method of \cite{renewal-opt-tac}. The 
ratio of expectations from the main method in \cite{renewal-opt-tac} requires knowledge of the probability distribution on $A[k]$. 
A heuristic is proposed in \cite{renewal-opt-tac} that uses a drift-plus-penalty minimization of $-v(R[k]-\theta[k-1]T[k]) + Q[k]Y[k]$, which has a simple 
decision complexity for each task that is the same as the decision complexity of the adaptive algorithm proposed in the current paper, and where $\theta[k]$ is defined as a running average: 
$$ \theta[k] = \frac{\sum_{i=1}^kR[i]}{\sum_{i=1}^kT[i]}$$
It is argued in \cite{renewal-opt-tac} that, if the heuristic converges, it converges to a point that is within $O(\epsilon)$ of optimality, where the parameter $v$ is chosen as $v=1/\epsilon$. We call this heuristic ``DPP with ratio averaging'' in the simulations.\footnote{Another method described in \cite{neely-renewal-jmlr} approximates the ratio of expectations using a window of $w$ past samples.  The per-task decision 
complexity grows with $w$ and hence is larger than the complexity of the algorithm proposed in the current paper. For ease of implementation, we have not considered this method.}  We also compare to a greedy method that removes any row $r$ of $A[k]$ that does not satisfy $Energy_r[k]/T_r[k]\leq 1/3$, and chooses from the remaining rows to maximize $R_r[k]/T_r[k]$. 

The i.i.d. matrices $\{A[k]\}$ have three columns and three rows of the form: 
$$ A[k] = \left[ \begin{array}{ccc}
1 & 0 & 0 \\
T_2[k] & R_2[k] & Y_2[k]\\
T_3[k] & R_3[k] & Y_3[k]
\end{array}
\right] $$
where $Y_r[k] = Energy_r[k]-(1/3)T_r[k]$ for $i\in\{2,3\}$. The first row corresponds ignoring task $k$ and remaining idle for $1$ unit of time, earning no reward  but using no energy, so $(T_1[k],R_1[k],Y_1[k])=(1,0,0)$. The second row corresponds to processing task $k$ at the home device. The third row corresponds to outsourcing task $k$ to a cloud device.  Two distributions are considered (specified at the end of this subsection). Under Distribution 1 the reward is the same for both rows 2 and 3, but the energy and durations of time are different. Under Distribution 2 the reward is higher for processing at the home device.  Both distributions have $t_{min}=1$, $t_{max}=12$, $r_{max}=20$. We use $\alpha = c_1/\max[c_2, 1/2]$ for the adaptive algorithm. Under the distributions used, the greedy algorithm is never able to use row 2, can always use either row 1 or 3, and always selects row 3.  

Figs. \ref{fig:Rewardnew} and \ref{fig:Powernew} consider reward and power for simulations over $5000$ tasks with i.i.d. $\{A[k]\}$ under Distribution 1. Fig. \ref{fig:Rewardnew} plots the running average of $\frac{\sum_{i=1}^k\expect{R[i]}}{\sum_{i=1}^k\expect{T[i]}}$ where expectations are attained by averaging over $40$ independent simulations. The horizontal asymptote illustrates the optimal $\theta^*$ as obtained by simulation. The simulation uses $v=50$ for the DPP with ratio averaging because this was sufficient for an accurate approximation of $\theta^*$, as seen in Fig. \ref{fig:Rewardnew}. The adaptive algorithm is considered for $v=10, v=50, v=200$. As predicted by our theorems, it can 
be seen that the converged reward is closer to $\theta^*$ as $v$ is increased (the case $v=200$ is competitive with DPP with ratio averaging). 
Fig. \ref{fig:Powernew} plots 
the corresponding running average of $\frac{\sum_{i=1}^k\expect{Energy[i]}}{\sum_{i=1}^k\expect{T[i]}}$.  The disadvantage of choosing a large value of $v$ is seen by the longer time required for time averaged power to converge to the horizontal asymptote $p_{av}=1/3$.  Figs. \ref{fig:Rewardnew} and \ref{fig:Powernew} show the  greedy algorithm has the worst reward per unit time and has average power significantly under the required constraint. This shows that, unlike the other algorithms, the greedy algorithm does not make intelligent decisions for using more power to improve its reward.  Considering only the performance shown in Figs. \ref{fig:Rewardnew} and \ref{fig:Powernew}, the DPP with ratio averaging heuristic demonstrates the best convergence times, which is likely due to the fact that it uses only one virtual queue $Q[k]$ while our adaptive algorithm uses $Q[k]$ and $J[k]$.  It is interesting to note that the adaptive algorithms and the DPP with ratio averaging heuristic both choose row 1 (idle) 
a significant fraction of time. This is because, when a task has a small reward but a large duration of time, it is better to throw the task away and wait idle for a short amount of time in order to see a new task with a hopefully larger reward.

\begin{figure}[t]
   \centering
   \includegraphics[width=4in]{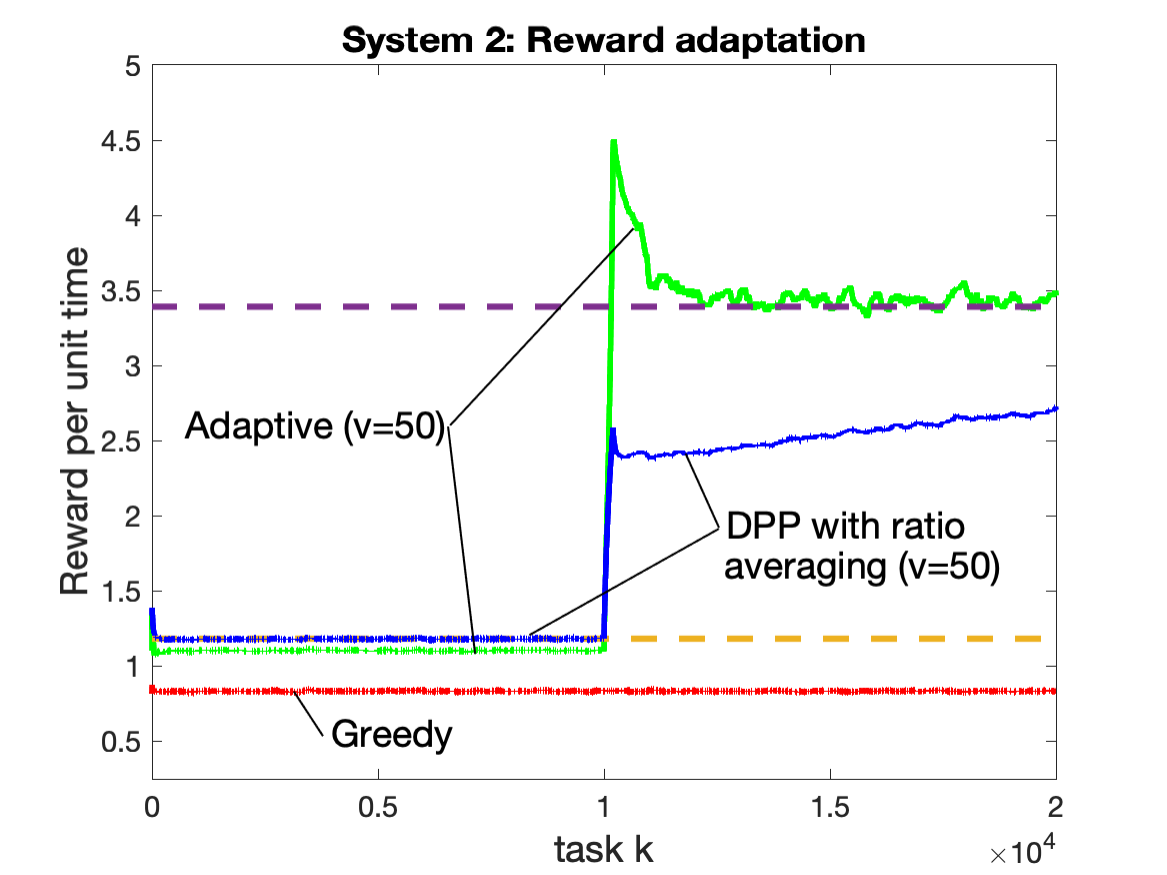} 
   \caption{Adaptation performance when the distribution is changed halfway through the simulation. Horizontal asymptotes are $\theta_1^*$ and $\theta_2^*$ for Distribution 1 and Distribution 2. The adaptive algorithm settles into the new optimality point $\theta_2^*$ while  DPP with ratio averaging cannot adapt.}
   \label{fig:adaptrewards}
\end{figure}

\begin{figure}[t]
   \centering
   \includegraphics[width=4in]{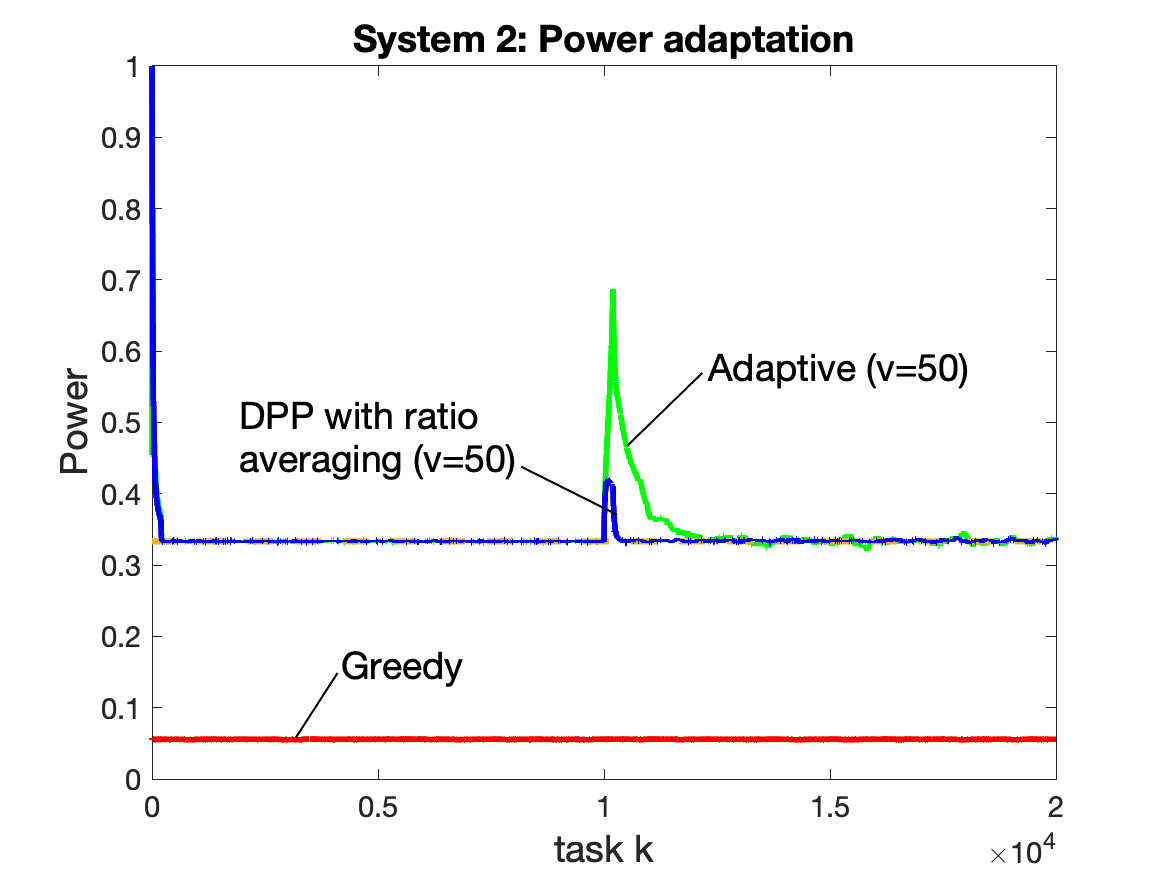} 
   \caption{Corresponding average power for the simulations of Fig. \ref{fig:adaptrewards}. The horizontal asymptote is $p_{av}=1/3$.}
   \label{fig:adaptpowers}
\end{figure}

The adaptation advantages of our proposed 
algorithm are illustrated in Figs.  \ref{fig:adaptrewards} and \ref{fig:adaptpowers}.
Both figures plot performance over a moving average with window size $w=200$ and average over $100$ independent simulations. 
The first half of the simulation uses i.i.d. $\{A[k]\}$ with Distribution 1, the second half uses Distribution 2.  The two horizontal asymptotes in Fig.  \ref{fig:adaptrewards} are the optimal $\theta_1^*$ and $\theta_2^*$ values for Distribution 1 and Distribution 2. As seen in the figure, both the adaptive algorithm and the DPP with ratio averaging heuristic quickly converge to the optimal $\theta_1^*$ value associated with Distribution 1 (the rewards under the adaptive algorithm are slightly less than that of the heuristic). At the time of change, the adaptive algorithm has a spike that lasts for roughly 2000 tasks until it settles down to $\theta_2^*$.  This can be viewed as the adaptation time, and can be decreased by decreasing the value of $v$ (at a corresponding accuracy cost). 
It converges from \emph{above} to $\theta_2^*$ because, as seen in Fig. \ref{fig:adaptpowers}, the spike marks a period of using more power than the required amount.   In contrast, the DPP with ratio averaging algorithm cannot adapt and never increases to the optimal value of $\theta_2^*$.

The distributions used are as follows: For each task $k$ there are two independent random variables $U_1[k], U_2[k] \sim Unif[0,1]$ generated. Then

\begin{itemize} 
\item Distribution 1: Note that $R_2[k]=R_3[k]$ and $T_2[k]<T_3[k]$ always.
\begin{align*}
&(T_2[k], R_2[k], Energy_2[k]) = (1+9U_1[k], 10U_1[k](U_2[k]+1), 1+9U_1[k])\\
&(T_3[k], R_3[k], Energy_3[k])=(6 + 6U_1[k], 10U_1[k](U_2[k]+1), U_1[k])
\end{align*}

\item Distribution 2: The $R_2[k]$ value is increased in comparison to Distribution 1.
\begin{align*}
&(T_2[k], R_2[k], Energy_2[k]) = (1+9U_1[k], \min\{20(U_2[k]+1), 20\}, 1+9U_1[k])\\
&(T_3[k], R_3[k], Energy_3[k])=(6 + 6U_1[k], 10U_1[k](U_2[k]+1), U_1[k])
\end{align*}
\end{itemize}

\subsection{Weight adjustment} 

While the $\Theta(1/\epsilon^2)$ adaptation times achieved by the proposed algorithm are asymptotically optimal, an important question is whether the coefficient can be improved by some constant factor. Specifically, this section attempts to reduce the 2000-task adaptation time seen in the spikes of Figs. \ref{fig:adaptrewards} and \ref{fig:adaptpowers} without degrading accuracy. We observe that the $J[k]$ and $Y[k]$ queues are weighted equally in the Lyapunov function $L[k]=\frac{1}{2}J[k]^2+\frac{1}{2}Y[k]^2$. More weight can be placed on $Y[k]$ to emphasize the average power constraint and thereby reduce the spike in Fig. \ref{fig:adaptpowers}. This can be done with no change in the mathematical analysis by redefining the penalty as $Y'[k]=wY[k]$ for some constant $w>0$. The constraint $\overline{Y}\leq 0$ is the same as $\overline{Y}'\leq 0$. We use $w=2$ and also double the $v$ parameter from 50 to 100, which maintains the same relative weight between reward and $Q[k]$ but deemphasizes the $J[k]$ virtual queue by a factor of 2.  

Figs. \ref{fig:graph7} and \ref{fig:graph8} plot performance over $3 \times 10^4$ tasks with Distribution 1 in the first third, Distribution 2 in the second third, and Distribution 1 in the final third.  The adaptive algorithm and the DPP with ratio averaging algorithm use the same parameters as in Figs. \ref{fig:adaptrewards} and \ref{fig:adaptpowers}.  The reweighted adaptive algorithm uses $v=100$ and $Y'[k]=2Y[k]$. It can be seen that the reweighting decreases adaptation time with no noticeable change in accuracy. This illustrates the benefits of weighting the power penalty $Y[k]$ more heavily than the virtual queue $J[k]$.

The simulations in Figs. \ref{fig:graph7} and \ref{fig:graph8} further show that the proposed adaptive algorithms can effectively handle multiple distributional changes. Indeed, the reward settles close to the new optimality point after each change in distribution.  In contrast, the DPP with ratio averaging algorithm, which was not designed to be adaptive, appears completely lost after the first distribution change and never recovers.  This emphasizes the importance of adaptive algorithms.

\begin{figure}[htbp]
   \centering
   \includegraphics[width=4in]{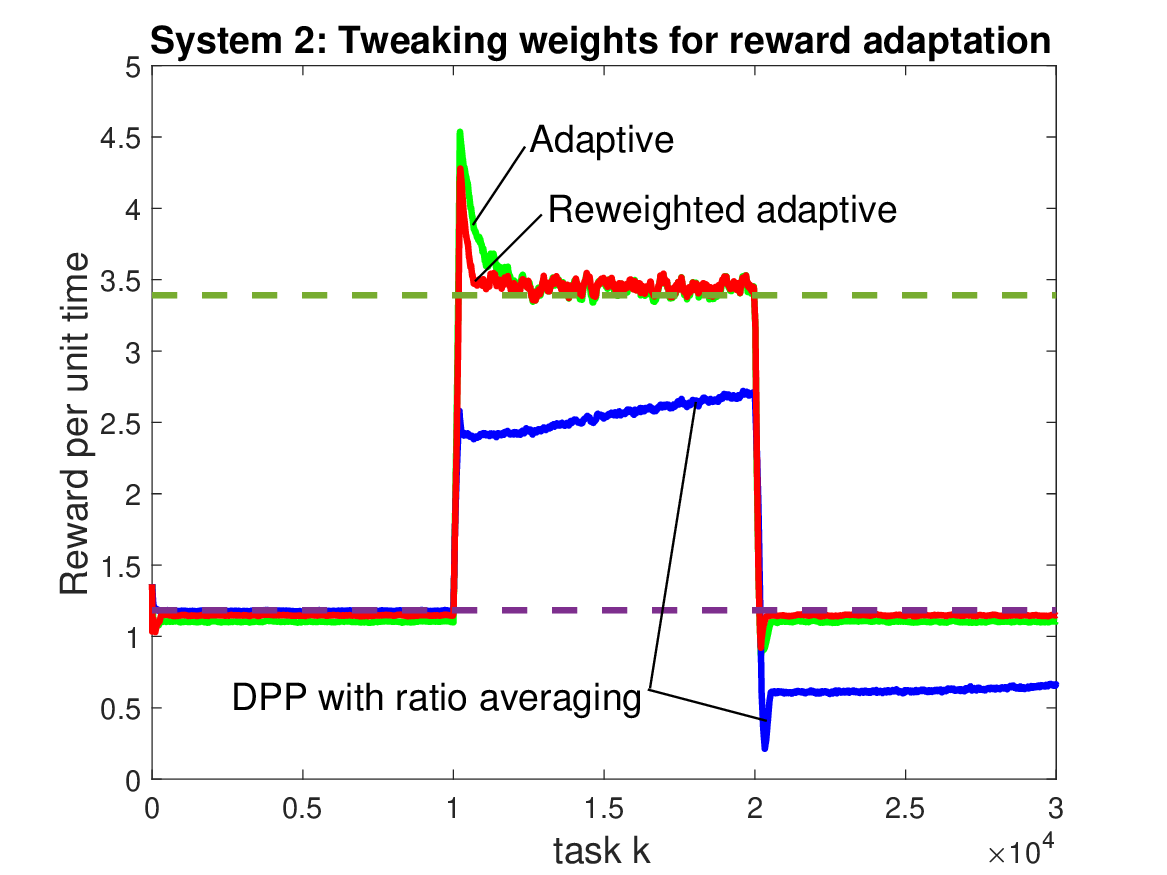} 
   \caption{Comparing the reward of the reweighted adaptive scheme over a simulation where the distribution is changed twice. Horizontal asymptotes are $\theta_1^*$ and $\theta_2^*$ for Distribution 1 and Distribution 2.  Parameters for the adaptive and DPP with ratio averaging algorithms are the same as for Figs. \ref{fig:adaptrewards} and \ref{fig:adaptpowers}.}
   \label{fig:graph7}
\end{figure}

\begin{figure}[htbp]
   \centering
   \includegraphics[width=4in]{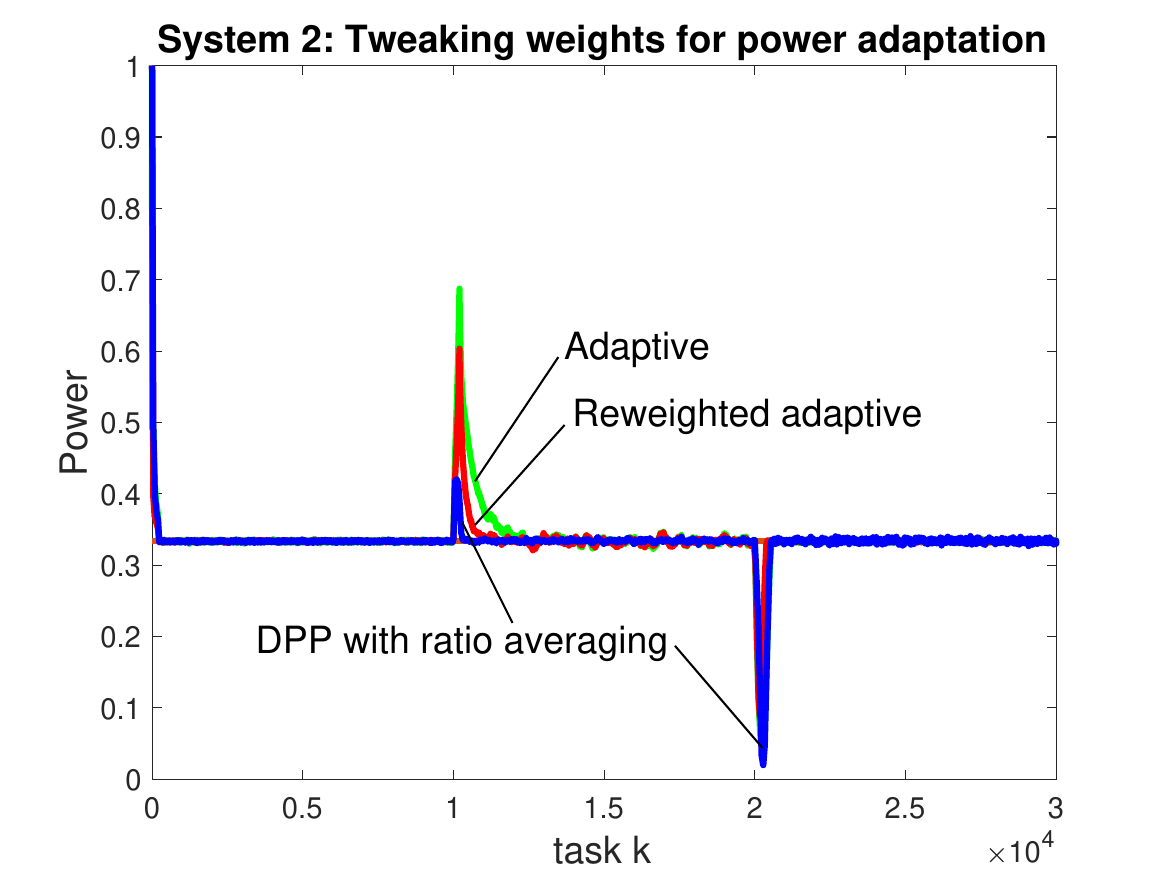} 
   \caption{Corresponding average power for the simulations of Fig. \ref{fig:graph7}. The horizontal asymptote is $p_{av}=1/3$.}
   \label{fig:graph8}
\end{figure}



\section{Conclusion} 

This paper gives an adaptive algorithm for renewal optimization, where decisions for each task $k$ determine the duration of the task, the reward for the task, and a vector of penalties. 
The algorithm operates without knowledge of system probabilities, has a low per-task decision complexity,
 and has asymptotic performance that achieves an optimal convergence time bound. A new
hierarchical decision rule enables the algorithm to achieve within $\epsilon$ of optimality over any sequence of $\Theta(1/\epsilon^2)$ tasks over which the probability distribution is fixed, regardless of system history.  The $\Theta(1/\epsilon^2)$ adaptation time matches prior converse results that show convergence time must be $\Omega(1/\epsilon^2)$ even when no penalty processes $Y[k]$ are considered and when convergence focuses on tasks $\{1, \ldots, m\}$ rather than on any consecutive sequence of $m$ tasks.

\section*{Appendix -- Details on Lemma \ref{lem:1}}

Part (a) of Lemma \ref{lem:1} uses arguments similar to those in \cite{sno-text}:  The definition of $\Gamma$ ensures any $(t,r,y)\in\Gamma$ can be realized as an \emph{unconditional expectation} $\expect{(T^*[k],R^*[k], Y^*[k])}$ under some particular conditional probability rule for choosing a row given the observed $A[k]$. The vector $(T^*[k], R^*[k], Y^*[k])$ can be made independent of $H[k]$ by using the independent $U\sim \mbox{Unif}[0,1]$ to make randomized decisions (the variable $U$ is defined in Section \ref{section:stochastic}). Independence implies that any version of the conditional 
expectation $\expect{(T^*[k],R^*[k], Y^*[k])|H[k]}$ is almost surely equal to the unconditional expectation. 

Part (b) of Lemma \ref{lem:1} uses concepts from Lemma 11 and Theorem 6 in \cite{neely-renewal-jmlr}: Define 
$$W[k]=(T[k], R[k], Y[k]) - \expect{(T[k],R[k],Y[k])|U, H[k]}$$  
Observe that for each $k>1$, $(W[1], \ldots, W[k-1])$ is $(U,H[k])$-measurable. Fix $k_1<k_2$. We have by iterated expectations: 
\begin{align} 
\expect{W[k_1]W[k_2]^{\top}} &= \expect{\expect{W[k_1]W[k_2]^{\top}|(U,H[k_2])}}\nonumber\\
&=\expect{W[k_1]\expect{W[k_2]^{\top}|(U,H[k_2])}}\label{eq:app1} \\
&=\expect{W[k_1]0^{\top}} \label{eq:app2} \\
&=0\nonumber
\end{align}
where \eqref{eq:app1} holds because $W[k_1]$ is $(U,H[k_2])$-measurable; \eqref{eq:app2} holds by definition of $W[k_2]$. 
Then $\{W[k]\}_{k=1}^{\infty}$ are bounded and uncorrelated random vectors, so 
\begin{equation} \label{eq:z-goes}
\lim_{m\rightarrow\infty} \frac{1}{m}\sum_{k=1}^mW[k]=0 \quad \mbox{(almost surely)} 
\end{equation} 
Since $(U,H[k])$ is independent of $A[k]$, arguments similar to Lemma 2 in \cite{neely-renewal-jmlr} imply 
$\expect{(T[k],R[k],Y[k])|U,H[k]}\in  \overline{\Gamma}$ almost surely for all $k$. Then, convexity of $\overline{\Gamma}$ implies (almost surely) 
$$ \frac{1}{m}\sum_{k=1}^m\expect{(T[k],R[k],Y[k])|U,H[k]} \in \overline{\Gamma} \quad \forall m \in \{1,2, 3, \ldots\} $$
Substituting the definition of $W[k]$ implies (almost surely):
$$ \frac{1}{m}\sum_{k=1}^m\left((T[k],R[k],Y[k])- W[k]\right) \in \overline{\Gamma} \quad \forall m \in\{1, 2, 3, \ldots\}$$
meaning the distance between $(\overline{T}[m], \overline{R}[m], \overline{Y}[m])$ and the set $\overline{\Gamma}$ 
is almost surely less than or equal to $\norm{\frac{1}{m}\sum_{k=1}^mW[k]}$, which  converges to 0 almost surely by \eqref{eq:z-goes}.

\section*{Acknowledgement} 

This work was supported in part by NSF grant SpecEES 1824418.

\bibliographystyle{unsrt}
\bibliography{../../../latex-mit/bibliography/refs}
\end{document}